\documentclass[12pt]{amsart}
\usepackage{mathpazo}
\usepackage{hyperref}
\usepackage{cleveref}
\usepackage{bm}
\usepackage{comment}
\usepackage[margin=1in]{geometry}
\usepackage{mathtools}
\usepackage{graphicx}
\usepackage[cmtip, all]{xy}
\usepackage{array}
\usepackage{amssymb,amsmath,latexsym,amsthm}

\newtheorem{theorem}{Theorem}[section]
\newtheorem{lemma}[theorem]{Lemma}
\newtheorem{proposition}[theorem]{Proposition}
\newtheorem{conjecture}[theorem]{Conjecture}

\newtheorem{corollary}[theorem]{Corollary}
\newtheorem{predefinition}[theorem]{Definition}
\newenvironment{definition}{\begin{predefinition}\rm}{\end{predefinition}}
\newtheorem{preremark}[theorem]{Remark}
\newenvironment{remark}{\begin{preremark}\rm}{\end{preremark}}
\newtheorem{prenotation}[theorem]{Notation}
\newenvironment{notation}{\begin{prenotation}\rm}{\end{prenotation}}
\newtheorem{preexample}[theorem]{Example}
\newenvironment{example}{\begin{preexample}\rm}{\end{preexample}}
\newtheorem{preclaim}[theorem]{Claim}

\newtheorem{prequestion}[theorem]{Question}

\newtheorem{preapplication}[theorem]{Application}
\newenvironment{application}{\begin{preapplication}\rm}{\end{preapplication}}

\numberwithin{equation}{section}

\newcommand \ZZ {{\mathbb Z}}
\newcommand \QQ {{\mathbb Q}}
\newcommand \NN {{\mathbb N}}
\newcommand  \FF {{\mathbb F}}
\newcommand \PP {{\mathbb P}^1}
\newcommand \CC {{\mathbb C}}
\newcommand \RR {{\mathbb R}}
\newcommand \hh {{\mathfrak h}}
\newcommand \CJ {{\mathcal J}}
\newcommand \cf {{\mathfrak f}}
\newcommand \CA {{\mathcal A}}
\newcommand \CM {{\mathcal M}}
\newcommand \CT {{\mathcal T}}
\newcommand \Sh {{\rm Sh}}
\newcommand \codim {{\rm codim}}
\newcommand \CX {{\mathcal X}}
\newcommand \co {{\mathfrak o}}
\newcommand \CO{{\mathfrak O}}
\newcommand \p{{\mathfrak p}}
\newcommand \QM {{\QQ[\mu_m]}}
\newcommand \sss {{\rm ss}}
\newcommand \ord {{{\rm ord}}}

\DeclareMathOperator{\GSp}{GSp}
\DeclareMathOperator{\length}{length}

\usepackage[usenames,dvipsnames]{color} 

\title{Newton polygon stratification of the Torelli locus in PEL-type Shimura varieties}
\date{}

\author{Wanlin Li}
\address{Department of Mathematics,
University of Wisconsin,
Madison, WI 53706, USA}
\email{wanlin@math.wisc.edu}

\author{Elena Mantovan}
\address{Department of Mathematics, California Institute of Technology, Pasadena, CA 91125, USA}
\email{mantovan@caltech.edu}

\author{Rachel Pries}
\address{Department of Mathematics, 
Colorado State University, 
Fort Collins, CO 80523, USA}
\email{pries@math.colostate.edu}

\author{Yunqing Tang}
\address{Department of Mathematics,
Princeton University,
Princeton, NJ 08540, USA}
\email{yunqingt@math.princeton.edu}

\begin{document}

\maketitle

\begin{abstract}
We study the intersection of the Torelli locus with the Newton polygon stratification of 
the modulo $p$ reduction of certain PEL-type Shimura varieties. 
We develop a clutching method to show that the intersection of the open Torelli locus with 
some Newton polygon strata is non-empty.
This allows us to give a positive answer, under some compatibility conditions, 
to a question of Oort about smooth curves in characteristic $p$ whose Newton polygons are an amalgamate sum.

As an application, we produce infinitely many new examples of Newton polygons that occur for 
smooth curves that are cyclic covers of the projective line. 
Most of these arise in inductive systems which demonstrate unlikely intersections of the 
open Torelli locus with the Newton polygon stratification in Siegel modular varieties.
In addition, for the
twenty special PEL-type Shimura varieties found in Moonen's work,
we prove that all Newton polygon strata intersect the open Torelli locus
(if $p>>0$ in the supersingular cases).


MSC10 classifications: primary 11G18, 11G20, 11M38, 14G10, 14G35; 
secondary 11G10, 14H10, 14H30, 14H40, 14K10
%
%
%
%
%
%
%
%
%
%
%
%
\end{abstract}

\section{Introduction} \label{Introduction}

\subsection{Overview}

Consider the moduli space ${\mathcal A}_g$ of principally polarized abelian varieties of dimension $g$ 
in characteristic $p > 0$.
It can be stratified by Newton polygon.
In general, it is unknown how these strata intersect the open Torelli locus ${\mathcal T}_g^\circ$, which is the image of the 
moduli space ${\mathcal M}_g$ of smooth genus $g$ curves under the Torelli morphism.
For a symmetric Newton polygon $\nu$ of height $2g$, in most cases it is not known 
whether the stratum ${\mathcal A}_g[\nu]$ intersects ${\mathcal T}_g^\circ$ or, 
equivalently, whether there exists a smooth
curve of genus $g$ in characteristic $p$ whose Jacobian has Newton polygon $\nu$.
This question is answered only when $\nu$ is close to ordinary, meaning that 
the codimension of ${\mathcal A}_g[\nu]$ in ${\mathcal A}_g$ is small.
 
In this paper, we develop a framework to study Newton polygons of Jacobians of 
$\mu_m$-covers of the projective line ${\mathbb P}^1$ for an integer $m \geq 2$.
We study the Newton polygon stratification on the Hurwitz spaces which represent such covers.
Using clutching morphisms, we produce singular curves
with prescribed Newton polygons.
Under an \emph{admissible} condition \Cref{defhp1}, 
these singular curves can be deformed to smooth curves 
which are $\mu_m$-covers of ${\mathbb P}^1$; 
we prove this can be done \emph{without changing the Newton polygon}
under a \emph{balanced} condition \Cref{def_hp2}, or a further \emph{compatible} condition \Cref{def_hp3}.

We then find systems of Hurwitz spaces of $\mu_m$-covers of ${\mathbb P}^1$
for which the admissible, balanced, and compatible conditions,
together with an expected codimension condition on the Newton polygon strata, 
can be verified inductively.
The base cases we use involve cyclic covers of ${\mathbb P}^1$ branched at $3$ points or the 20 special families found by Moonen \cite{moonen}.

As an application, 
we find numerous infinite sequences of unusual Newton polygons which occur for Jacobians of smooth curves.
Most of these arise in an \emph{unlikely intersection} of the open Torelli locus ${\mathcal T}_g^\circ$
with the Newton polygon strata of ${\mathcal A}_g$ \Cref{def_unlikely}.

In essence, 
our strategy is to generalize earlier techniques for 
studying the intersection of the Torelli locus with the Newton polygon stratification 
from the context of the system of moduli spaces ${\mathcal A}_g$ to the context of appropriate
inductive systems of PEL-type Shimura varieties. 
Hurwitz spaces of $\mu_m$-covers of $\PP$ determine unitary Shimura varieties
associated with the group algebra of $\mu_m$, 
as constructed by Deligne--Mostow \cite{delignemostow}.
By work of Kottwitz \cite{kottwitz1,kottwitz2}, Viehmann--Wedhorn \cite{wedhornviehmann}, and Hamacher \cite{hamacher},
the Newton polygon stratification of the modulo $p$ reduction of 
these Shimura varieties is well understood in terms of 
their signature types and the congruence class of $p$ modulo $m$. 
It is difficult to understand the intersection of the Newton polygon stratum $Sh[\nu]$ with the open Torelli locus 
because the codimension of the latter in $Sh$ grows with $g$;
our main results verify 
that this intersection is non-trivial
when $\nu$ is close to $\mu$-ordinary, for infinitely many Shimura varieties $Sh$ of PEL-type.

\subsection{Comparison with earlier results}

In 2005, Oort proposed the following conjecture. 
\begin{conjecture} \label{Coortconj} (\cite[Conjecture 8.5.7]{Oort05})
	For $i=1,2$, let $g_i\in\ZZ_{\ge 1}$ and let $\nu_i$ be a symmetric Newton polygon appearing on ${\mathcal T}_{g_i}^\circ$. Write $g=g_1+g_2$. 
	Let $\nu$ be the amalgamate sum of $\nu_1$ and $\nu_2$. 
	Then $\nu$ appears on ${\mathcal T}_g^\circ$.
\end{conjecture}

It is not clear whether Oort's conjecture is true in complete generality.  
Theorems \ref{thm_muord} and \ref{differentm} show that Oort's conjecture has an affirmative answer in many cases.

The results in Section \ref{Sclutch} can be viewed as a generalization of Bouw's work 
\cite{Bouwprank} about the intersection of ${\mathcal T}_g^\circ$ with the stratum of maximal $p$-rank in a 
PEL-type Shimura variety.
For most families of $\mu_m$-covers and most congruence classes of $p$ modulo $m$, 
the maximal $p$-rank does not determine the Newton polygon.

Clutching morphisms were also used to study the intersection 
of ${\mathcal T}_g^\circ$ with the $p$-rank stratification of ${\mathcal A}_g$
in Faber and Van der Geer's work \cite[Theorem 2.3]{FVdG}; also \cite{AP:mono}, \cite{glasspries}, \cite{AP:hyp}.

The results in Section \ref{sec_nonord} generalize Pries' work \cite[Theorem 6.4]{priesCurrent}, which states that
if a Newton polygon $\nu$ occurs on ${\mathcal M}_g$ 
with the expected codimension, then the Newton polygon $\nu \oplus (0,1)^n$ occurs on ${\mathcal M}_{g+n}$  
with the expected codimension for $n \in \ZZ_{\geq 1}$. However,  the expected codimension condition is difficult to verify for most Newton polygons $\nu$.

\subsection{Outline of main results}\label{sec_intro_clut}

In Section \ref{Snotation}, we review key background about
Hurwitz spaces, PEL-type Shimura varieties, and Newton polygon stratifications.
In \Cref{AdmissibleClutch}, we analyze the image of a clutching morphism $\kappa$
on a pair of $\mu_m$-covers of ${\mathbb P}^1$.
	
In Section \ref{Sclutch}, we study whether the open Torelli locus ${\mathcal T}_g^\circ$ intersects the $\mu$-ordinary 
Newton polygon stratum, see \Cref{Dmuord}, inside the Shimura variety $S$. The first main result \Cref{thm_muord} 
provides a method to leverage information about this question from lower to higher genus. 
 Under a balanced condition on the signatures \Cref{def_hp2}, we can determine the $\mu$-ordinary Newton polygon as the Shimura variety varies in the clutching system (\Cref{balanced}, which we prove in \Cref{sec_muord}). 

The most ground-breaking results in the paper are in Section \ref{sec_nonord}, where 
we study the intersection of the open Torelli locus ${\mathcal T}_g^\circ$
with the non $\mu$-ordinary Newton polygon strata inside the Shimura variety $S$.
Theorem \ref{differentm} also provides a method
to leverage information from lower to higher genus.  Under an additional 
compatibility condition on the signatures \Cref{def_hp3}, we can control the codimension
of the Newton polygon strata as the Shimura variety varies in the clutching system (\Cref{codim}). 

In Sections \ref{Sinfclut} and \ref{sec_infclut_nu}, we find situations 
where Theorems \ref{thm_muord} and \ref{differentm} can be implemented recursively, 
infinitely many times, which yields smooth curves with arbitrarily large genera and 
prescribed (unusual) Newton polygons. 
We do this by constructing 
suitable \emph{infinite clutching systems} of PEL-type Shimura varieties which satisfy the
admissible, balanced, (and compatible) conditions at every level. 
For example, we prove:

\begin{theorem}[Special case of \Cref{inf-muord}]\label{main-thm-ord}
	Let $\gamma = (m,N,a)$ be a monodromy datum as in \Cref{Dmonodatum}.  
	Let $p$ be a prime such that $p \nmid m$.
	Let $u$ be the $\mu$-ordinary Newton polygon associated to $\gamma$ as in Definition \ref{Dmuord}.
	Suppose there exists a cyclic cover of $\PP$ defined over $\overline{\FF}_p$ with monodromy datum $\gamma$ 
	and Newton polygon $u$.\footnote{See \Cref{muordoccur} for when this condition is satisfied.} 	Then, for any $n\in \ZZ_{\geq 1}$, there exists a smooth curve over $\overline{\mathbb{F}}_p$ with Newton polygon $u^n\oplus (0,1)^{(m-1)(n-1)}$.\footnote{The slopes of this Newton polygon are the slopes of $u$ (with multiplicity scaled by $n$) and $0$ and $1$ each with multiplicity $(m-1)(n-1)$.}
\end{theorem}

For a symmetric Newton polygon $\nu$ of height $2g$, 
we say that the open Torelli locus has an \emph{unlikely intersection} with the Newton polygon stratum 
${\mathcal A}_g[\nu]$ in ${\mathcal A}_g$ if there exists a smooth curve of genus $g$ with Newton polygon $\nu$
and if ${\rm dim}({\mathcal M}_g) < {\rm codim}({\mathcal A}_g[\nu], {\mathcal A}_g)$, \Cref{def_unlikely}.
In Section \ref{unlikely}, we study 
the asymptotic of ${\rm codim}({\mathcal A}_g[\nu], {\mathcal A}_g)$ for the Newton polygons 
$\nu$ appearing in Sections \ref{Sinfclut} and \ref{sec_infclut_nu}, and 
verify that most of our inductive systems
produce unlikely intersections once $g$ is sufficiently large, for most congruence classes of $p$ modulo $m$.

\subsection{Applications}

In \Cref{MSS} in \Cref{MoonenSS}, we prove that all the
Newton polygons for the Shimura varieties associated to the families in \cite[Table 1]{moonen} 
occur for smooth curves in the family.

In \Cref{Applications1}, we construct explicit infinite sequences of 
Newton polygons that occur at odd primes for smooth curves which demonstrate unlikely intersections. 
For example, by \Cref{main-thm-ord} applied to $\gamma=(m,3, (1,1,m-2))$, we prove:  

\begin{application} \label{Agenp} (\Cref{3gn-2g}) 
	Let $m \in \ZZ_{> 1}$ be odd and $h=(m-1)/2$.  
	Let $p$ be a prime, $p \nmid 2m$, such that the order $f$ of $p$ in $(\ZZ/m\ZZ)^*$ is even and $p^{f/2}\equiv -1 \bmod m$.
	For $n \in \ZZ_{\ge 1}$,
	there exists a $\mu_m$-cover $C \to {\mathbb P}^1$ defined over $\overline{\mathbb{F}}_p$
		where $C$ is a smooth curve of genus $g=h(3n-2)$ 
		with Newton polygon $\nu=(1/2, 1/2)^{h n}\oplus (0,1)^{2h(n-1)}$.\footnote{Its slopes are 
		$1/2$ with multiplicity $2hn$ and $0$ and $1$ each with multiplicity $2h(n-1)$.} 
		If $n \ge 34/h$, then ${\rm Jac}(C)$ lies in the unlikely intersection ${\mathcal T}^\circ_g \cap 
		{\mathcal A}_{g}[\nu]$.  
\end{application}

In \Cref{2/3ForAllg}, we apply \Cref{Agenp} when $m=3$ to verify, 
for $p \equiv 2 \bmod 3$ and $g \in \ZZ_{\geq 1}$, 
there exists a smooth curve of genus $g$ defined over $\overline{\FF}_p$
whose Newton polygon only has slopes $\{0, 1/2, 1\}$ and the multiplicity of slope 
$1/2$ is at least $2\lfloor g/3 \rfloor$.
To our knowledge, this is the first time for any odd prime $p$ that 
a sequence of smooth curves has been produced for every $g \in \ZZ_{\geq 1}$ such that the multiplicity of the 
slope $1/2$ in the Newton polygon grows linearly in $g$.

\subsection*{Acknowledgements} 
Pries was partially supported by NSF grant DMS-15-02227. Tang is partially supported by NSF grant DMS-1801237.
We thank the Banff International Research Station for hosting \emph{Women in Numbers 4},
the American Institute of Mathematics for supporting our square proposal, and
an anonymous referee for valuable suggestions.


\section{Notations and Preliminaries} \label{Snotation}

More details on this section can be found in \cite[Sections 2,3]{LMPT2} and 
\cite[Sections 2,3]{moonen}.

\subsection{The group algebra of $m$-th roots of unity}
Let $m,d \in \ZZ_{\geq 1}$.
Let $\mu_m:=\mu_m(\CC)$ denote the group of $m$-th roots of unity in $\CC$.
Let $K_d$ be the $d$-th cyclotomic field over $\QQ$.
Let $\QM$ denote the group algebra of $\mu_m$ over $\QQ$. Then $\QM=\prod_{d|m}K_d$. We endow $\QM$ with the involution $*$ 
induced by the inverse map on $\mu_m$, i.e., $\zeta^*:= \zeta^{-1}$ for all $\zeta\in\mu_m$.  

Set $\CT:=\hom_\QQ (\QM,\CC)$.  If $W$ is a $\QM\otimes_\QQ\CC$-module, we write $W=\oplus_{\tau\in\CT} W_\tau$, where $W_\tau$ denotes the subspace of $W$ on which $a\otimes 1\in\QM\otimes_\QQ\CC$ acts as $\tau(a)$. We fix an identification $\CT=\ZZ/m\ZZ$ by defining, for all $n\in \ZZ/m\ZZ$,
\[\tau_n(\zeta):=\zeta^n, \text{ for all }\zeta\in \mu_m.\]

Let $m\geq 1$.  For $p\nmid m$,
we identify $\CT=\hom_\QQ(\QM, \overline{\QQ}_p^{\rm un})$, where $\overline{\QQ}_p^{\rm un}$ is the maximal unramified extension of $\QQ_p$ in an algebraic closure. There is a natural action of the Frobenius $\sigma$ on $\CT$, 
defined by $\tau\mapsto \tau^\sigma:=\sigma\circ \tau$. 
Let $\CO$ be the set of all $\sigma$-orbits $\co$ in $\CT$.  

\subsection{Newton polygons}
Let $X$ be an abelian scheme 
defined over the algebraic closure $\FF$ of $\FF_p$.  
Then there is a finite field $\FF_0/\FF_p$, 
an abelian scheme $X_0/\FF_0$, and $\ell \in \ZZ_{\geq 1}$, 
such that $X \simeq X_0 \times_{\FF_0} \FF$ and the action of $\sigma^\ell$  on $H^1_{\rm cris}(X_0/W(\FF_0))$ is linear; here $W(\FF_0)$ denotes the Witt vector ring of $\FF_0$. 
The \emph{Newton polygon} $\nu(X)$ of $X$ is the multi-set of rational numbers $\lambda$
such that $\ell \lambda$ are the valuations at $p$ of the eigenvalues of   $\sigma^\ell$  acting on $H^1_{\rm cris}(X_0/W(\FF_0))$; 
the Newton polygon does not depend on the choice of $(\FF_0, X_0,\ell)$.

The {\em $p$-rank} of $X$ is the multiplicity of the slope $0$ in $\nu(X)$; it equals $\dim_{\FF_p}(\hom(\mu_p,X))$.

\label{NPsum}
If $\nu_1$ and $\nu_2$ are two Newton polygons, the {\em amalgamate sum} $\nu_1\oplus \nu_2$ is 
the disjoint union of the multi-sets $\nu_1$ and $\nu_2$.
	We denote by $\nu^d$  the amalgamate sum of $d$ copies of $\nu$.
	
The Newton polygon $\nu(X)$ is typically drawn as a lower convex polygon, with slopes $\lambda$ occurring with multiplicity $m_\lambda$, where $m_\lambda$ denotes the multiplicity of $\lambda$ in the multi-set.
The Newton polygon of a $g$-dimensional abelian variety is symmetric, 
with endpoints $(0,0)$ and $(2g,g)$, integral break points, and slopes in $\QQ\cap [0,1]$. 
For convex polygons, 
we write $\nu_1\geq \nu_2$ if $\nu_1,\nu_2$ share the same endpoints and $\nu_1$ lies below $\nu_2$. 

We denote by $\ord$ the Newton polygon  $(0,1)$ and by $\sss$ the Newton polygon $(1/2,1/2)$. 
For $s,t\in \ZZ_{\geq 1}$, with $s\leq t/2$ and ${\rm gcd}(s,t)=1$, 
we write $(s/t, (t-s)/t)$ for the Newton polygon with slopes $s/t$ and  $(t-s)/t$, each with multiplicity $t$. 

Suppose $Y$ is a semi-abelian scheme defined over $\FF$. 
Then $Y$ is an extension of an abelian scheme $X$ by a torus $T$; its Newton polygon is  
$\nu(Y):=\nu(X)\oplus\ord^\epsilon$, where $\epsilon={\rm dim}(T)$.

\subsection{Cyclic covers of the projective line}\label{ZmNa}

\begin{definition} \label{Dmonodatum}
	Fix integers $m\geq 2$, $N\geq 3$ and an $N$-tuple of integers $a=(a(1),\dots, a(N))$. 
	Then $a$ is an {\em inertia type} for $m$ and  
	$\gamma=(m,N,a)$ is a {\em  monodromy datum} if
	\begin{enumerate}
		\item $a(i)\not\equiv 0 \bmod m$, for each $i=1, \dots, N$, 
		\item $\gcd(m, a(1),\dots, a(N))=1$, 
		\item $\sum_ia(i)\equiv 0 \mod m$.
	\end{enumerate}
\end{definition}

For later applications, we sometimes consider a {\em generalized monodromy datum}, 
in which we allow $a(i) \equiv 0 \bmod m$.  In the case that $a(i)=0$, we set ${\rm gcd}(a(i), m) = m$.

Two monodromy data $(m,N,a)$ and $(m',N',a')$ are {\em equivalent} 
if $m=m'$, $N=N'$, and the images of $a,a'$ in $(\ZZ/m\ZZ)^N$ 
are in the same orbit under $(\ZZ/m\ZZ)^*\times {\rm Sym}_N$. 

For fixed $m$, we work over an irreducible scheme over $\ZZ[1/m, \zeta_m]$.
Let $U\subset ({\mathbb A}^1)^N$ be the complement of the weak diagonal.
Consider the smooth projective curve $C$ over $U$ whose fiber at each point $t=(t(1),\dots, t(N))\in U$ has affine model 
\begin{equation} \label{EformulaC}
y^m=\prod_{i=1}^N(x-t(i))^{a(i)}.
\end{equation}
Consider the $\mu_m$-cover $\phi: C \to {\mathbb P}^1_U$ defined by the function $x$ and the
$\mu_m$-action $\iota: \mu_m \to {\rm Aut}(C)$ given by 
$\iota(\zeta) \cdot (x,y)=(x,\zeta\cdot y)$ for all $\zeta\in\mu_m$.

For a closed point $t\in U$, the cover $\phi_t: C_t\to \PP$ is a $\mu_m$-cover, 
branched at $N$ points $t(1),\dots , t(N)$ in $\PP$, and 
with local monodromy $a(i)$ at $t(i)$.
By the hypotheses on the monodromy datum, $C_t$ is a geometrically irreducible curve of genus $g$, where
\begin{equation} \label{Egenus}
g=g(m,N,a)=1+\frac{1}{2}\Big((N-2)m-\sum_{i=1}^N\gcd(a(i),m)\Big).
\end{equation}

Take $W=H^0(C_t, \Omega^1)$ and,
under the identification $\CT=\ZZ/m\ZZ$, let $\cf(\tau_n) = {\rm dim}(W_{\tau_n})$.
The \emph{signature type} of $\phi$ is defined as $\cf=(\cf(\tau_1), \ldots, \cf(\tau_{m-1}))$.
By \cite[Lemma 2.7, \S3.2]{moonen}, 
\begin{equation}\label{DMeqn}
\cf(\tau_n)=\begin{cases} -1+\sum_{i=1}^N\langle\frac{-na(i)}{m}\rangle & \text{ if $n\not\equiv 0 \bmod m$}\\
0 & \text{ if $n\equiv 0 \bmod m$}.\end{cases}
\end{equation}
where, for any $x\in \RR$, $\langle x\rangle$ denotes the fractional part of $x$.
The signature type of $\phi$ does not depend on $t$; it determines and is uniquely determined by the inertia type.

\subsection{Hurwitz spaces}

Let ${\mathcal M}_g$ be the moduli space of smooth curves of genus $g$ in characteristic $p$.
Its Deligne--Mumford compactification $\overline{\mathcal M}_g$ is the moduli space of stable curves of genus $g$.
For a Newton polygon $\nu$, let ${\mathcal M}_g[\nu]$
be the subspace whose points represent objects with Newton polygon $\nu$. 
We use analogous notation for other moduli spaces.

We refer to \cite[Sections 2.1-2.2]{AP:trielliptic} for a more complete description of 
Hurwitz spaces for cyclic covers of ${\mathbb P}^1$.
Consider the moduli functor $\overline{\CM}_{\mu_m}$ (resp.\ $\tilde{\CM}_{\mu_m}$)
on the category of schemes over $\ZZ[1/m, \zeta_m]$; 
its points represent admissible stable $\mu_m$-covers $(C/U, \iota)$ of a genus $0$ curve
(resp.\ together with an ordering of the smooth branch points and 
the choice of one ramified point above each of these). 
We use a superscript $\circ$ to denote the subspace of points for which $C$ is smooth.
By \cite[Lemma 2.2]{AP:trielliptic}, $\overline{\CM}_{\mu_m}$ (resp.\ $\tilde{\CM}_{\mu_m}$) 
is a smooth proper Deligne--Mumford stack and $\overline{\CM}^\circ_{\mu_m}$ 
(resp.\ $\tilde{\CM}^\circ_{\mu_m}$) is open and dense within it.

For each irreducible component of $\tilde{\CM}_{\mu_m}$, 
the monodromy datum $\gamma=(m,N, a)$ 
of the $\mu_m$-cover $(C/U, \iota)$ is constant.  Conversely, the substack $\tilde{\CM}^\gamma_{\mu_m}$ 
of points representing $\mu_m$-covers with monodromy datum $\gamma$ is irreducible, 
\cite[Corollary 7.5]{fultonhur}, \cite[Corollary 4.2.3]{wewersthesis}.

On $\overline{\CM}_{\mu_m}$, there is no ordering of the ramification points; 
so only the unordered multi-set $\overline{a}= \{a(1), \ldots, a(N)\}$ is well-defined. 
The components of $\overline{\CM}_{\mu_m}$ are indexed by $\overline{\gamma}=(m, N, \bar{a})$.
By \cite[Lemma 2.4]{AP:trielliptic}, the forgetful morphism 
$\tilde{\CM}^\gamma_{\mu_m} \to \overline{\CM}_{\mu_m}^{\bar{\gamma}}$
is \'etale and Galois.

\begin{definition} \label{DzmNa}\footnote{This definition is slightly different from the one in our previous papers \cite{LMPT}, \cite{LMPT2}.}
	If $\gamma= (m,N,a)$ is a monodromy datum, let $\tilde{Z}(\gamma)=\tilde{\mathcal M}_{\mu_m}^\gamma$ and let 
	$\overline{Z}(\gamma)$ 
	be the reduced image of $\tilde{\CM}^\gamma_{\mu_m}$ in $\overline{\CM}_g$.
	We denote the subspace representing objects where $C/U$ is smooth 
	(resp.\ of compact type\footnote{A stable curve has \emph{compact type} if its dual graph is a tree.
The Jacobian of a stable curve $C$ is a semi-abelian variety; 
also $C$ has compact type if and only if ${\rm Jac}(C)$ is an abelian variety.}, resp.\ stable) by
	\[Z^\circ(\gamma) \subset Z(\gamma) \subset \overline{Z}(\gamma)
		\text{ and }
 	\tilde{Z}^\circ(\gamma)\subset  \tilde{Z}^c(\gamma)\subset \tilde{Z}(\gamma).\] 
\end{definition}

		By definition, $\overline{Z}(\gamma)$ is a reduced irreducible proper substack of $\overline{\CM}_g$.
	It depends uniquely on the equivalence class of $\gamma$. 	
	The forgetful morphism $\tilde{Z}(\gamma)\rightarrow \overline{Z}(\gamma)$ is finite and hence it preserves the dimension of any substack.

\begin{remark}\label{plus1}
Let $\gamma'=(m,N',a')$ be a generalized monodromy datum. 
Assume that $a'(N') \equiv 0 \bmod m$ and $a'(i) \not \equiv 0 \bmod m$ for $1 \leq i < N'$.
Consider the monodromy datum $\gamma=(m, N'-1,a)$, where $a(i)=a'(i)$ for $1\leq i \leq N'-1$.
Then $\overline{Z}(\gamma')=\overline{Z}(\gamma)$ and $\tilde{Z}(\gamma')= \tilde{Z}(\gamma)_1$, 
where the subscript $1$ indicates that the data includes one marked point on the curve.
The fibers of the morphism $\tilde{Z}^\circ(\gamma')\rightarrow {Z}^\circ(\gamma')$ are of pure dimension $1$.
\end{remark}

\subsection{Shimura varieties associated to monodromy data} \label{shimura}
Consider $V:=\QQ^{2g}$ endowed with the standard symplectic form $\Psi:V\times V\to\QQ$  and $G:={\rm GSp}(V,\Psi)$, the group of symplectic similitudes. 
Let $(G_\QQ,\hh)$ be the Siegel Shimura datum.

Fix $x\in Z(\gamma)(\CC)$ and let $(\CJ_x,\theta)$ denote the Jacobian of the curve represented by $x$ together with its principal polarization $\theta$.  Choose a symplectic similitude 
\[\alpha: (H_1(\CJ_x,\ZZ), \psi_\theta)\to (V ,\Psi)\]
where $\psi_\theta$ denotes the Riemannian form on $H_1(\CJ_x,\QQ)$ corresponding to $\theta$.
Via $\alpha$, the $\QQ[\mu_m]$-action on $\CJ_x$ induces a $\QM$-module structure on $V$, and the Hodge decomposition of $H_1(\CJ_x,\CC)$ induces
a $\QM\otimes_\QQ\CC$-linear decomposition 
$V_\CC=V^+\oplus V^-$.

We recall the PEL-type Shimura stack $\Sh(\mu_m,\cf)$ given in \cite{delignemostow}.
The Shimura datum of $\Sh(\mu_m,\cf)$ given by $(H,\hh_\cf)$ is defined as 
\[H:=GL_\QM(V)\cap \GSp(V,\Psi),\] and $\hh_{\cf}$ the $H$-orbit in $\{h\in \hh \mid h \text{ factors through }H\}$ determined by the isomorphism class of the $\QM\otimes_\QQ\CC$-module $V^+$,
i.e., by the integers
$\cf(\tau):=\dim_\CC(V^+_\tau)$, for all $\tau\in\CT$.
Under the identification $\CT=\ZZ/m\ZZ$, 
the formula for $\cf(\tau_n)$ is that given in \eqref{DMeqn}. 

For a Shimura variety $\Sh:=\Sh(H,\hh_\cf)$ of PEL type, we use $\Sh^*$ to denote the Baily-Borel (i.e., minimal) compactification and $\overline{\Sh}$ to denote a toroidal compactification (see \cite{Lan}).

The Torelli morphism $T: {\mathcal M}^c_g \to {\mathcal A}_g$
takes a curve of compact type to its Jacobian.

\begin{definition} \label{Dspecial}
	We say that $Z(\gamma)$ 
	is {\it special} if $T(Z(\gamma))$ is open and closed in the PEL-type Shimura stack $\Sh(\mu_m,\cf)$ given in \cite{delignemostow} (see \cite[Section 3.3]{LMPT2} for details).
\end{definition}

	If $N=3$, then $T(Z(\gamma))$ is a point of $\CA_g$ representing an abelian variety with complex multiplication and 
	is thus special, \cite[Lemma 3.1]{LMPT}.
By \cite[Theorem 3.6]{moonen},
if $N\geq 4$, then $Z(\gamma)$ is special 
	if and only if $\gamma$ is equivalent to one of twenty examples in \cite[Table 1]{moonen}.

\subsection{The Kottwitz set and the $\mu$-ordinary Newton polygon} \label{Skottwitz}

Let $p\nmid m$ be a rational prime. Then the Shimura datum $(H,\hh_\cf)$ is unramified at $p$. We write  $H_{\QQ_p}$ for the fiber of $H$ at $p$, and $\mu_\hh$ for the conjugacy class of $p$-adic cocharacters $\mu_h$ associated with $h\in\hh_\cf$. 

Following \cite{kottwitz1}-\cite{kottwitz2}, we denote by $B(H_{\QQ_p},\mu_\hh)$ the partially ordered 
set of $\mu$-admissible $H_{\QQ_p}$-isocrystal structures on $V_{\QQ_p}$. 
By \cite[Theorem 1.6]{wedhornviehmann} (see also \cite{wedhorn}), $B(H_{\QQ_p},\mu_\hh)$ 
can be canonically identified with the set of Newton polygons appearing on $\Sh(H,\hh)$.\footnote{The term Newton polygon usually refers to the image of an element in $B$ via the Newton map. Here, since we work with PEL-type Shimura varieties of type A and C, the Newton map is injective and hence we do not distinguish between the elements in $B$ and the corresponding Newton polygons.} We sometimes write  $\Sh:=\Sh(H,\hh)$ and 
$B=B(\Sh):=B(H_{\QQ_p},\mu_\hh)$. 

\begin{definition} \label{Dmuord}
	The {\em $\mu$-ordinary} Newton polygon $u:=u_{\mu-ord}$ is
	 the unique maximal element (lowest Newton polygon) of $B(H_{\QQ_p},\mu_\hh)$.
\end{definition}

An explicit formula for $u$ is given below; (see also \cite[Section 4.1-4.2]{LMPT2}).

\subsubsection{Formula for slopes and multiplicities} \label{Sformulau}
	Let $\cf$ be a signature type. Fix an orbit $\co$ in $\CT$.
	We recall the formulas from \cite[Section 1.2.5]{moonen} for the slopes and multiplicities of the $\co$-component $u(\co)$ of the $\mu$-ordinary Newton polygon in terms of $\cf$, following the notation in 
	\cite[Section 2.8]{eischenmantovan}, \cite[Section 4.2]{LMPT2}. 
	
	With some abuse of notation, we replace $\CT$ by $\CT-\{\tau_0\}$ and $\CO$ by $\CO-\{\{\tau_0\}\}$.
Let $g(\tau):=\dim_\CC(V_\tau)$.
As the integer $g(\tau)$ depends only on the order of $\tau$ in the additive group $\ZZ/m\ZZ$, and thus only on the orbit $\co$ of $\tau$, we sometimes write $g(\co)=g(\tau)$, for any/all $\tau\in\co$. 

\begin{remark} \label{Rdual}
For all $\tau\in \CT$, $\dim_\CC(V^+_{\tau^*})=\dim_\CC(V^-_{\tau})$, and thus $\cf(\tau)+\cf(\tau^*)=g(\tau)$.  \end{remark}

	Let $s=s(\co)$ be the number of distinct values of 
	$\{\cf(\tau) \mid \tau\in\co\}$ in the range $[1,g(\co)-1]$.  
	We write these distinct values as
	\[g(\co)> E(1)> E(2)>\cdots >E(s)>0.\] 
	Let $E(0):=g(\co)$ and $E({s+1}):=0$.
	Then $u(\co)$ 
	has exactly $s+1$ distinct slopes, denoted by 
	$0\leq \lambda(0) < \lambda(1) <\cdots <\lambda(s)\leq 1$.
	For $0\leq t\leq s$, the $(t+1)$-st slope is
	\begin{equation}\label{eq_slopes}
	\lambda(t):=\frac{1}{|\co|} \#\{\tau \mid \cf_i(\tau) \ge E(t)\}.
	\end{equation}
	The  slope $\lambda(t)$
	occurs in $u(\co)$ with multiplicity 
	\begin{equation}\label{eq_multi}
	\rho(\lambda(t)):=|\co|(E(t)-E(t+1)).
	\end{equation} 
	
\subsection{Geometry of the Newton polygon strata on $\Sh$}
	
For $b\in B$, let
$\Sh[b]:=\Sh(H,\hh)[b]$ denote the corresponding Newton polygon stratum in $\Sh$. 
In other words, $\Sh[b]$ is the locally closed substack of $\Sh$ parametrizing abelian schemes with Newton polygon $b$. 
By Hamacher \cite[Theorem 1.1, Corollary 3.12]{hamacher}, based on the work of Chai \cite{Chai}, Mantovan \cite{Mantovan}, and Viehmann \cite{Viehmann}, and Kottwitz \cite[Section 8]{kottwitz92},\footnote{Hamacher proved that $Sh[b]$ is non-empty and equidimensional of expected dimension. Since Hecke translations preserve the Newton polygon strata and act transitively on the irreducible components of $Sh$, we deduce the same result for $S[b]$. See 
\cite[\S 8]{kottwitz92} for a more detailed discussion.} on each irreducible component $S$ of $\Sh$, the substack $S[b]$ is non-empty and equidimensional and
\begin{eqnarray}\label{eqn_codim}
\codim (S[b],\Sh)= {\rm length}(b),
\end{eqnarray}
where ${\rm length}(b) = {\rm max}\{n \mid \ {\rm there} \ {\rm exists} \ {\rm a} \ {\rm chain} \ 
b=\nu_0 < \nu_1 < \cdots < \nu_n=u, \  \nu_i \in B\}$. 

The Newton stratification extends to the toroidal and minimal compactifications $\overline{\Sh}, \Sh^*$.  In \cite[\S 3.3]{Lan-Stroh}, the authors studied the Newton stratification on compactifications of PEL-type Shimura varieties at good primes. They proved in this case
that all the Newton strata are so called \emph{well-positioned} subschemes \cite[Proposition 3.3.9]{Lan-Stroh}. 
In particular, by 
\cite[Definition 2.2.1, Theorem 2.3.2]{Lan-Stroh}, the set of Newton polygons on (each irreducible component of) $\overline{\Sh}$ is the same as that on $\Sh$ and, for any $b \in B$,
\begin{eqnarray}\label{eqn_codim_cpt}
\codim (\overline{\Sh}[b],\overline{\Sh}) =\codim(\Sh[b],\Sh).
\end{eqnarray}

By the next remark, there exists a $\mu_m$-cover of smooth curves
having monodromy datum $\gamma$ and $\mu$-ordinary Newton polygon $u$ if there exists such a cover of stable curves.

\begin{lemma}\label{muordsmooth}
	The following are equivalent:
	$Z^\circ(\gamma)[u]$ is non-empty;
	$Z^\circ(\gamma)[u]$ is open and dense in $Z(\gamma)$; 
 and $Z(\gamma)[u]$ is non-empty.	
\end{lemma}

\begin{proof}
This is clear because the Newton polygon is lower semi-continuous, $Z(\gamma)$ is irreducible, and  
	$Z^\circ(\gamma)$ is open and dense in $Z(\gamma)$.
\end{proof}

\begin{remark}
	The Ekedahl--Oort type is also determined for many of the smooth curves in this paper.  
	The reason is that the $\mu$-ordinary Newton polygon stratum in these PEL-type Shimura varieties coincides with the unique open Ekedahl--Oort stratum.
	Hence one may compute the Ekedahl--Oort type of these smooth curves using \cite[Section 1.2.3]{Moonen04}. 
\end{remark}


\section{Clutching morphisms}\label{AdmissibleClutch}

\subsection{Background} \label{Sindprec}

We study a generalization of the clutching morphisms, 
for curves that are $\mu_m$-covers of $\PP$.
The clutching morphisms are the closed immersions \cite[3.9]{knudsen2}: 
\[\kappa_{i,g-i}: \overline{\CM}_{i;1} \times \overline{\CM}_{g-i;1} \to \overline{\CM}_{g} \ {\rm and} \  
\lambda: \overline{\CM}_{g-1;2} \to \overline{\CM}_g,\]
for $1 \leq i \leq g-1$.
The image of $\kappa_{i,g-i}$ is the component $\Delta_{i}$ of the boundary of $\CM_g$, 
whose generic point represents a stable curve that has two components, of genus $i$ and $g-i$, intersecting in one point.
The image of $\lambda$ is the component $\Delta_0$ of the boundary of $\CM_g$, whose points represent 
stable curves that do not have compact type.

Given a pair of cyclic covers of ${\mathbb P}^1$, we analyze the image of a clutching morphism $\kappa$, which 
shares attributes of both $\kappa_{i, g-i}$ and $\lambda$.
	To provide greater flexibility, we include cases when the covers have different degrees
	or when the two covers are clutched together at several points.  
	As a result, a curve in the image of $\kappa$ may not have compact type.

\begin{notation} \label{Nind}
Let $\gamma=(m,N,a)$ be a monodromy datum. For an integer $d \geq 1$, consider the induced monodromy data 
$\gamma^{\dagger_d} =(dm,N,da)$, which we sometimes denote $\gamma^\dagger$.\end{notation}

If $\phi: C \to \PP$ is a $\mu_m$-cover with monodromy datum $\gamma$, then
${\rm Ind}_m^{dm}(\phi): {\rm Ind}_m^{dm} C \to \PP$ is a $\mu_{dm}$-cover with 
the induced monodromy datum $\gamma^{\dagger_d}$.
The signature type of ${\rm Ind}_m^{dm}(\phi)$ is 
$\cf^{\dagger_d}=\cf\circ\pi_d$, where $\pi_d:\ZZ/dm\ZZ\to\ZZ/m\ZZ$ denotes the natural projection. By definition, the $\mu$-ordinary polygon of $\Sh(\mu_{dm},\cf^{\dagger_d})$ is $u^{\dagger_d}=u^d$.

\subsection{Numerical data and hypothesis (A)}\label{AdmissibleData}

\begin{notation} \label{Nadmissible}
	Fix integers $m_1,m_2\geq 2$, $N_1, N_2\geq 3$. Let $m_3={\rm lcm}(m_1,m_2)$.
	For $i=1,2$: let $d_i=m_3/m_i$;
	let $a_i=(a_i(1),\dots,a_i(N_i))$ be such that $\gamma_i=(m_i,N_i,a_i)$ is a 
	(generalized) monodromy datum; and let $g_i=g(m_i,N_i,a_i)$ as in \eqref{Egenus}. 
\end{notation}

\begin{definition}\label{defhp1}
	A pair of monodromy data $\gamma_1=(m_1,N_1, a_1)$, $\gamma_2=(m_2,N_2,a_2)$ as in \Cref{Nadmissible} is {\em admissible} 
	if it satisfies \[{\rm hypothesis (A)}: \ d_1 a_1(N_1)+d_2a_2(1)\equiv 0\bmod m_3.\]
\end{definition}

\begin{notation} \label{Dtype3}
	Assume hypothesis (A) for the pair $\gamma_1, \gamma_2$.
	Set $r_1 = \gcd (m_1, a_1(N_1))$, and $r_2=\gcd(m_2, a_2(1))$. 
	Let $r_0 = \gcd(r_1,r_2)$ and  
	let
	\begin{equation} \label{Dep}
	\epsilon=d_1d_2r_0-d_1-d_2+1 \ {\rm and} \ g_3=d_1g_1+d_2g_2+\epsilon. 
	\end{equation}
\end{notation}

Note that $d_1r_1=d_2r_2=d_1d_2r_0$ since ${\rm gcd}(d_1,d_2) = 1$. 

\begin{definition} \label{Dmd3}
	If $\gamma_1, \gamma_2$ is an admissible pair of (generalized) monodromy data, 
	we define $\gamma_3=(m_3, N_3, a_3)$ by 
	$m_3:={\rm lcm}(m_1,m_2)$, $N_3:=N_1+N_2-2$ and the $N_3$-tuple $a_3$ as \[a_3(i):=\begin{cases}
	d_1a_1(i) \text{ for $1\leq i\leq N_1-1$,}\\ d_2a_2(i-N_1+2) \text{ for $N_1\leq i\leq N_1+N_2-2$.}\end{cases}\] 
	\end{definition}
	
\begin{lemma} \label{Lmdgf}
The triple $\gamma_3$ from \Cref{Dmd3} is a (generalized) monodromy datum.
If $\phi_3: C \to \PP$ is a cover with monodromy datum $\gamma_3$,
then the genus of $C$ is $g_3$ as in \eqref{Dep}.  
\end{lemma}

\begin{proof} Immediate from \Cref{Dmonodatum} and \eqref{Egenus}. 
\end{proof}

The signature type for $\phi_3$ is given in \Cref{Dsig3}, see \Cref{Lsig3}.

\begin{remark} \label{Rmodifyslightly}
A pair $\gamma_1, \gamma_2$ of non-admissible monodromy data can be modified slightly to produce a pair $\gamma_1', \gamma_2'$ of admissible generalized monodromy data by marking an extra unramified fiber.  
Specifically, let \begin{enumerate}
\item $\gamma'_1=(m_1,N_1+1, a'_1)$ with $a'_1(i)=a_1(i)$ for $1\leq i\leq N_1$,  and $a'_1(N_1+1)=0$;
\item $\gamma'_2=(m_2,N_2+1,a'_2)$ with
$a'_2(1)=0$ and
$a'_2(i)=a_2(i-1)$ for $2\leq i\leq N_2+1$. 
\end{enumerate}
This does not change the geometry, because $Z(\gamma'_i)=Z(\gamma_i)$ for $i=1,2$ by \Cref{plus1}.
\end{remark}

\subsection{Clutching morphisms for cyclic covers} \label{Sclutchall}

\begin{notation}
Let $\gamma_1=(m_1,N_1, a_1)$, $\gamma_2=(m_2,N_2,a_2)$ be an admissible pair of monodromy data 
as in \Cref{Nadmissible}.
Let $\gamma_3=(m_3, N_3, a_3)$
be the monodromy datum from \Cref{Dmd3}.
For $i=1,2,3$, let $\tilde{Z}_i=\tilde{Z}(\gamma_i)$ be as in \Cref{DzmNa}.
\end{notation}

Recall that the points of $\tilde{Z}^\circ_3$ represent
		$\mu_{m_3}$-covers $C \to \PP$ with monodromy datum $\gamma_3$, where $C$ is smooth.
		The next result is more general than related results in the literature.

\begin{proposition} \label{Pclutchadd}
	If hypothesis (A) (the admissible condition) is satisfied, there is a clutching morphism
		$\kappa: \tilde{Z}_{1} \times \tilde{Z}_{2} \to \tilde{Z}_3$,
		and the image of $\kappa$ is in the boundary of $\tilde{Z}^\circ_3$. 
		\end{proposition}

\begin{proof}
	For $i=1,2$, let $\phi_i:C_i \to {\mathbb P}^1$ be the 
	$\mu_{m_i}$-cover with $N_i$ ordered and labeled $\mu_{m_i}$-orbits of points which is represented by a point 
	of $\tilde{Z}_{i}$. 
	There is a natural inclusion $\mu_{m_i} \subset \mu_{m_3}$. 
	Let $C^{\dagger}_i = {\rm Ind}_{m_i}^{m_3}(C_i)$ be the induced curve
	and let $\phi^{\dagger}_i: C^{\dagger}_i \to {\mathbb P}^1$ be the induced cover.
	It has inertia type $a^\dagger_i=d_ia_i = (d_ia_i(1),\dots, d_ia_i(N_i))$. 
	
	We define the morphism $\kappa$ on the pair $(\phi_1, \phi_2)$.
		There are $d_1r_1=d_1{\rm gcd}(m_1,a_1(N_1))$ points of $C^\dagger_1$ above $t_1(N_1)$.
		By hypothesis (A), this equals $d_2r_2=d_2{\rm gcd}(m_2, a_2(1))$, the number of points 
		of $C^\dagger_2$ above $t_2(1)$.
		Let $C_3$ be the curve whose components are the $d_1$ components of $C^\dagger_1$ 
		and the $d_2$ components of $C^\dagger_2$, formed by 
		identifying these labeled $d_1r_1=d_2r_2$ points in ordinary double points,
		in a $\mu_{m_3}$-equivariant way.
		
		Then $C_3$ is a $\mu_{m_3}$-cover of a tree $P$ of two projective lines. It has $N_3$ labeled 
		$\mu_{m_3}$-orbits with inertia type $a_3$ and is thus
		represented by a point of $\tilde{Z}_3$.
		The admissible condition in \Cref{defhp1} is exactly the (local) admissible condition on the 
		covers $\phi^\dagger_1$ and $\phi^\dagger_2$ at the point(s) above $t_1(N_1)$ and $t_2(1)$.  
		By \cite[2.2]{ekedahlhurwitz}, the $\mu_{m_3}$-cover
		$C_3 \to P$ is in the boundary of $\tilde{Z}^\circ_3$ if and only if hypothesis (A) is satisfied.
		\end{proof}

The curve $C_3$ constructed in \Cref{Pclutchadd} is a $\mu_{m_3}$-cover of type $\gamma_3$		
	and thus has arithmetic genus $g_3$ by \Cref{Lmdgf}.	

\begin{proposition} \label{PNPC3}
	The curve $C_3$ constructed in the proof of \Cref{Pclutchadd} has compact type if and only if $\epsilon = 0$.
	It has Newton polygon
\begin{eqnarray}\label{EnuC3}
	\nu(C_3) & = & \nu(C_1)^{d_1} \oplus \nu(C_2)^{d_2} \oplus  \ord^{\epsilon}. \qquad 
\end{eqnarray}
	\end{proposition}

The term $\ord^{\epsilon}$ can be viewed as the {\em defect} of $\nu(C_3)$.
It measures the number of extra slopes of $0$ and $1$ that arise when $C_3$ does not have 
compact type.  
By \Cref{Dtype3},
$\epsilon = 0$ if and only if $r_0=1$ and either $d_1=1$ or $d_2=1$.

\begin{proof}
	By construction,
	the dual graph of $C_3$ is a bipartite graph, 
	with $d_1$ (resp.\ $d_2$) vertices in bijection with the components of $C^\dagger_1$ (resp.\ $C^\dagger_2$).
	In $C_3$, each of the components of $C^\dagger_1$ intersects each of the components of $C^\dagger_2$ 
	in $r_0$ points.
	After removing $d_1d_2(r_0-1)$ edges from the dual graph, there is a unique edge
	between each pair of vertices on opposite sides.
	After removing another $(d_1-1)(d_2-1)$ edges from the dual graph, it is a tree.
	Thus the Euler characteristic of the dual graph of $C_3$ is 
	$d_1d_2(r_0-1)+(d_1-1)(d_2-1)$, which equals $\epsilon$.	
	In particular, $C_3$ has compact type if and only if $\epsilon = 0$.
	
	By \cite[9.2.8]{BLR}, for some torus $T$ of rank $\epsilon$, there is a short exact sequence
	\[0 \to T \to {\rm Jac}(C_3) \to {\rm Jac}(C_1)^{d_1} \oplus {\rm Jac}(C_2)^{d_2} \to 0.\]
	Since ${\rm dim}_{\FF_p}(\mu_p, T)=\epsilon$, the Newton polygon of
	${\rm Jac}(C_3)$ is the amalgamate sum (\Cref{NPsum}) of the Newton polygon of 
	${\rm Jac}(C_1)^{d_1} \oplus {\rm Jac}(C_2)^{d_2}$ and $\ord^{\epsilon}$, which yields \eqref{EnuC3}.
\end{proof}

\subsection{The signature}

We find the signature $\cf_3$ for a cover with monodromy datum $\gamma_3$.

\begin{definition}\label{delta}
	Let $d,R \in \ZZ_{\geq 1}$ with $dR|m$.
	For $n\in\ZZ/m\ZZ$, we define
	$\delta_{d, dR}(n):=1$ if $dRn\equiv 0\bmod m$ and $dn\not\equiv 0\bmod m$, and $\delta_{d, dR}(n):=0$ otherwise. 
	\end{definition}
	
	Equivalently, if $\tau\in\CT$, then
	$\delta_{d, dR}(\tau)=1$ if the order of $\tau$ in $\ZZ/m \ZZ$ divides $dR$ but not $d$, and $\delta_{d, dR}(\tau)=0$ otherwise.
	Since $\delta_{d, dR}(\tau)$ only depends on the orbit $\co$ of $\tau$, we also write  $\delta_{d, dR}(\co):=\delta_{d, dR}(\tau)$, for any/all $\tau\in\co$.
	
\begin{definition} \label{Dsig3}
	Let $\delta:=\delta_{d,dr_0}+\delta_{d_1,d_1d_2}-\delta_{1,d_2}$, for $d=d_1d_2$.
	For $n\in\ZZ/m_3\ZZ - \{0\}$, let
	\begin{eqnarray}\label{signaturesum}\label{eq_signature}
	\cf_3(\tau_n)=\cf^\dagger_1(\tau_n)+\cf^\dagger_2(\tau_n)+ \delta(n).
	\end{eqnarray}
\end{definition}

By definition, $\delta(n)=1$ if  $d_1d_2r_0n\equiv 0\bmod m_3$ and $d_1n\not\equiv 0\bmod m_3$ and $d_2n\not \equiv 0\bmod m_3$, and $\delta(n)=0$ otherwise.

\begin{lemma} \label{Lsig3}
If $\phi_3: C \to \PP$ is a cover with (generalized) monodromy datum $\gamma_3$, as defined in \Cref{Dmd3},
then the signature type of $C_3$ is $\cf_3$.  
\end{lemma}

\begin{proof}
We use \eqref{DMeqn} to compute $\cf_3$.\footnote{Alternatively, one may deduce the formula for $\cf_3$ geometrically 
since the extra term $\delta$ records the $\ZZ/m_3 \ZZ$-action on the dual graph of the curve $C_3$ 
constructed in \Cref{Pclutchadd}.}
If $n\equiv 0\bmod m_1$, then 
$\cf_1^\dagger(\tau_n)=0$ and $\cf_3(\tau_n)=\cf^\dagger_2(\tau_n)$.
If $n\equiv 0\bmod m_2$, then $\cf_2^\dagger(\tau_n)=0$ and $\cf_3(\tau_n)=\cf^\dagger_1(\tau_n)$.
For $n\in\ZZ/m_3\ZZ$, with $n\not\equiv 0\bmod m_1$ and $n\not\equiv 0\bmod m_2$, then
	\[(\cf^\dagger_1(\tau_n)+1)+(\cf^\dagger_2(\tau_n)+1)-(\cf_3(\tau_n)+1)=\langle \frac{-nd_1a_1(N_1)}{m_3}\rangle + \langle \frac{-nd_2a_2(1)}{m_3}\rangle.\]
	The right hand side is $0$ or $1$; it is $0$ if and only if $nd_1a_1(N_1)\equiv nd_2a_2(1)\equiv 0 \bmod m_3$.
\end{proof}

\subsection{Compatibility with Shimura variety setting}

\begin{notation} \label{Nshimclutch}
Fix an admissible pair $\gamma_1=(m_1,N_1, a_1)$, $\gamma_2=(m_2,N_2,a_2)$ of monodromy data as in \Cref{Nadmissible}. 
Consider the monodromy datum $\gamma_3$ as in \Cref{Dmd3}. In particular, let $m_3={\rm lcm}(m_1,m_2)$
and let $\cf_3$ be as in \Cref{Dsig3}.

For each $i=1,2,3$, let $Z_i:=Z(m_i, N_i, {a_i})$, 
and similarly $Z_i^\circ$, $\tilde{Z}^c_i$, etc as in Definition \ref{DzmNa}. 
Let $\Sh_i:=\Sh_i(\mu_{m_i}, \cf_i)$ denote the Shimura substack of ${\mathcal A}_g$ as in \Cref{shimura}.  
Let $\CX_i$ be the universal abelian scheme over $\Sh_i$, $B_i:=B(\Sh_i)$ the set of Newton polygons of $\Sh_i$, 
and $u_i$ the $\mu$-ordinary Newton polygon in $B_i$ from \Cref{Dmuord}.
\end{notation}

Via the Torelli map $T$, the clutching morphism $\kappa: \tilde{Z}^c_1 \times \tilde{Z}^c_2 \to \tilde{Z}_3$ 
is compatible with a morphism into the minimal compactification of the Shimura variety
\begin{eqnarray*}
\iota & : & \Sh_1\times \Sh_2\rightarrow \Sh^*_3 \qquad 
\end{eqnarray*} where $\iota(\CX_1,\CX_2):=\CX_1^{d_1} \oplus \CX^{d_2}_2.$\footnote{An abelian variety of dimension less than $g_3$ with $\mu_m$-action can be viewed as a point on the boundary of $\Sh^*_3$ as it comes from the pure part of some semi-abelian variety of dimension $g_3$ with $\mu_m$-action, which is a point on $\overline{\Sh}_3$. The image of the Torelli map in the minimal compactification is determined by the Torelli map on 
the irreducible components of the curve and forgets the dual graph structure.}
If $\epsilon=0$ then ${\rm Im}(\iota)$ lies in $\Sh_3$ and the reader may focus on this case;
if $\epsilon \not = 0$, then ${\rm Im}(\iota)$ is contained in the boundary $\Sh_3^* - \Sh_3$.

By \Cref{PNPC3}, $\iota(\Sh_1[\nu_1],\Sh_2[\nu_2])  \subseteq  \Sh^*_3( \nu_1^{d_1}\oplus  \nu^{d_2}_2\oplus \ord^{\epsilon})$, which yields:

\begin{lemma}\label{clutchNP}
	If $\nu_i\in B_i$, then $\nu_1^{d_1} \oplus  \nu_2^{d_2} \oplus  \ord^{\epsilon} \in B_3$.   
	In particular,
	$u_3\geq u_1^{d_1}\oplus  u_2^{d_2} \oplus \ord^{\epsilon}$. 
\end{lemma}

In Proposition \ref{balanced}, we give a necessary and sufficient condition on the pair of signature types $(\cf_1,\cf_2)$  for the equality $u_3=u_1^{d_1}\oplus  u_2^{d_2} \oplus \ord^{\epsilon}$ to hold.


\section{The Torelli locus and the $\mu$-ordinary locus of Shimura varieties} \label{Sclutch}

In this section, we prove theorems about the intersection of the open Torelli locus with the 
$\mu$-ordinary Newton polygon stratum in a PEL-type Shimura variety.  
The main result, Theorem \ref{thm_muord}, provides a method 
to leverage information from smaller dimension to larger dimension.
This provides an inductive method to prove that the open Torelli locus intersects the $\mu$-ordinary stratum for certain types of families.

In Section \ref{Sinfclut},
we use the main theorem to 
establish the existence of smooth curves of arbitrarily large genus with prescribed Newton polygon, 
see \Cref{inf-ord,infinite-ord,inf-muord,twofamilyinfinite}.
For the base cases of the inductive method, we can use any instances when the $\mu$-ordinary Newton polygon is known to occur (see \Cref{muordoccur}).

\begin{remark} 
The method in this section does not give results for every monodromy datum $\gamma$.
For example, it is not known whether the $\mu$-ordinary Newton polygon occurs on 
	$Z^\circ(\gamma)$ for all $p \equiv 3,5 \bmod 7$ when $\gamma=(7,4, (1,1,2,3))$.  
	In this case, $\cf=(2,1,1,1,1,0)$, ${\rm dim}(Z(\gamma)) = 1$, and ${\rm dim}(S(\gamma))=2$.
	The three Newton polygons on $S(\gamma)$ are $(1/6, 5/6)$,  $(1/3, 2/3)^2$, and $ss^6$, which all have $p$-rank $0$.
	None of the degenerations for this family satisfy hypothesis (B) as defined below.
\end{remark}

\subsection{Hypothesis (B)}\label{ShypB}

We fix an admissible pair $\gamma_1=(m_1,N_1, a_1)$, $\gamma_2=(m_2,N_2,a_2)$ of (generalized) monodromy data as in \Cref{Nadmissible}.  
We fix a prime $p$ such that $p \nmid m_3={\rm lcm}(m_1,m_2)$ and work over $\overline{\FF}_p$.
Recall Notation \ref{Dtype3} and \ref{Nshimclutch}.
So $d_i=m_3/m_i$ and $\cf_i^\dagger:= \cf_i\circ\pi_{i}$ for $i=1,2$. 

\begin{definition}\label{def_hp2}
	The pair of monodromy data $\gamma_1, \gamma_2$ is {\em balanced} if,
	for each orbit $\co \in \CT=\hom_\QQ(\QQ[\mu_{m_3}],\CC)$ 
	and all $\omega,\tau\in\co$, the values of the induced signature types satisfy:
	\[{\rm hypothesis (B):} \ \text{if } \cf^\dagger_1(\omega)> \cf^\dagger_1(\tau) \text{ then }\cf^\dagger_2(\omega)\geq \cf^\dagger_2(\tau); \text{ if }\cf^\dagger_2(\omega)> \cf^\dagger_2(\tau) \text{ then }\cf^\dagger_1(\omega)\geq \cf^\dagger_1(\tau). \]
\end{definition}

\begin{remark}\label{rmk_hp2}
	\begin{enumerate}
	\item If $m_1=m_2$, then hypothesis (B) is symmetric for the pair $\gamma_1, \gamma_2$.
	\item
		If $\gamma_1=\gamma_2$, then hypothesis (B) is automatically satisfied.
		\item
		Hypothesis (B) depends (only) on the congruence of $p$ modulo $m_3$.
		If $p\equiv 1 \bmod m_3$, then each orbit $\co$ in $\CT$ has size one and hypothesis (B) is vacuously satisfied. 
		\item 
		Let $\gamma_3$ be as in \Cref{Dmd3}.
		If the pair $\gamma_1, \gamma_2$ satisfies hypotheses (A) and (B),
		then $\gamma_i, \gamma_3$ satisfies hypothesis (B) for $i=1$ and for $i=2$.
	\end{enumerate}
\end{remark}

\Cref{balanced} below gives a geometric interpretation of hypothesis (B).
Recall that $\epsilon = d_1d_2r_0 - d_1 - d_2 + 1$ from \Cref{Dtype3}.

\begin{proposition}\label{chain2}\label{balanced}
	Let $\gamma_1, \gamma_2$ be an admissible pair of monodromy data. 
	Consider the monodromy datum $\gamma_3 $ as in \Cref{Dmd3}.
	Then the equality $u_3=u_1^{d_1}\oplus  u_2^{d_2}\oplus \ord^\epsilon$
		holds if and only if the pair $\gamma_1,\gamma_2$ is balanced.	
\end{proposition}

We postpone the proof of \Cref{balanced} to the independent \Cref{sec_muord}.

\subsection{A first main result}\label{induction-statement}

In this subsection, we assume that the pair $\gamma_1, \gamma_2$ 
is admissible and balanced, meaning that it
satisfies hypotheses (A) and (B) as in Definitions \ref{defhp1} and \ref{def_hp2}.
Let $\gamma_3=(m_3,N_3,a_3)$ and $\cf_3$ be as in Definitions \ref{Dmd3} and \ref{Dsig3}.

The next result provides a partial positive answer to \Cref{Coortconj} when $\epsilon = 0$.

\begin{theorem}\label{thm_muord}
	Let $\gamma_1, \gamma_2$ be an admissible, balanced pair of monodromy data.
	If $Z_1^\circ[u_1]$ and $Z_2^\circ[u_2]$ are both non-empty, then $Z_{3}^\circ[u_3]$ is non-empty.
\end{theorem}

\begin{proof}
By \Cref{Lmdgf,Lsig3}, the signature for the monodromy datum $\gamma_3$ is given in \eqref{eq_signature}.
	By \Cref{balanced}, hypothesis (B) implies that 
	$u_3=u_1^{d_1}\oplus  u^{d_2}_2\oplus \ord^\epsilon$.
	From the hypothesis, $\tilde{Z}_1^\circ[u_1]$ and $\tilde{Z}_2^\circ[u_2]$ are both non-empty.
	By \Cref{Pclutchadd}, the image of $\kappa$ on $\tilde{Z}_{1}^\circ[u_1] \times \tilde{Z}_{2}^\circ[u_2]$ is in $\tilde{Z}_{3}$.
	By Proposition \ref{PNPC3}, the Newton polygon of a curve $C_3$ 
	represented by a point in the image of $\kappa$ 
	is given by $\nu(C_3)=u_1^{d_1}\oplus u^{d_2}_2 \oplus \ord^\epsilon$, which is $u_3$.
	Thus $Z_3[u_3]$ is non-empty and applying \Cref{muordsmooth} finishes the proof.
\end{proof}

\subsection{Infinite clutching for $\mu$-ordinary} \label{Sinfclut}

In this section, 
we find situations in which Theorem \ref{thm_muord} 
can be implemented recursively, infinitely many times, to 
verify the existence of smooth curves of arbitrarily large genus with prescribed Newton polygon. 
The required input is a family (or a compatible pair of families) of cyclic covers of ${\mathbb P}^1$ 
for which the $\mu$-ordinary Newton polygon at a prime $p$ is known to occur (see \Cref{muordoccur}). 
Concrete implementations of these results can be found in \Cref{Applications1}.

\subsubsection{Base cases}\label{basecases}

We recall instances when the the $\mu$-ordinary Newton polygon $u$ is known to occur for the Jacobian of a 
(smooth) curve in the family $Z$.

\begin{proposition}\label{muordoccur}
		The $\mu$-ordinary Newton polygon stratum $Z(m,N,a)[u]$ is non empty if either:
	\begin{enumerate}
		\item $N=3$; or
		\item $N\geq 4$ and
		$(m,N,a)$ is equivalent to one of the twenty examples in \cite[Table 1]{moonen}; or
		\item $u$ is the only Newton polygon in $B(H_{\QQ_p},\mu_\hh)$ of maximum $p$-rank
		and either $N=4$ or $p\geq m(N-3)$ or $p \equiv \pm 1 \bmod m$. 
			\end{enumerate}
\end{proposition}
\begin{proof}
	\begin{enumerate}
		\item When $N=3$, then $Z$ is 0-dimensional and thus special as in \Cref{Dspecial}.
		
		\item This is \cite[Proposition 5.1]{LMPT2}.
		
		\item 
		For any monodromy datum $\gamma = (m,N,a)$, define 
\begin{equation}\label{eq_prank}
\beta(\gamma):=\sum_{\tau\in\CT} \min_{j\in \NN}\{\cf( \tau^{\sigma^j})\} = 
\sum_{\co \in \CO} \#\co \cdot {\rm min}\{\cf(\tau) \mid \tau \in \co\}. 
\end{equation}
By \cite[Equation (1)]{Bouwprank}, $\beta(\gamma)$ is an upper bound for the 
$p$-rank of curves in $Z^\circ(\gamma)$.
By \cite[Theorem 6.1, Propositions 7.7, 7.4, 7.8]{Bouwprank}, if $p \geq m(N-3)$ or $N=4$ or $p \equiv \pm 1 \bmod m$,
	then there exists a $\mu_m$-cover $C \to {\mathbb P}^1$ defined over 
	$\overline{\FF}_p$ with monodromy datum $\gamma$, 
	for which the $p$-rank of $C$ equals $\beta(\gamma)$. 
	The $p$-rank is the multiplicity of $1$ as a slope of the Newton polygon.
	From the formulas \eqref{eq_slopes} and \eqref{eq_multi} for the slopes and multiplicities 
	of the $\mu$-ordinary Newton polygon $u=u(\gamma)$,
	one can check 
that the $p$-rank of $u$ equals $\beta(\gamma)$.
		\qedhere
	\end{enumerate}
\end{proof}

We refer to \cite{LMPT} for the computation of the $\mu$-ordinary polygon $u$ when $N=3$, and to \Cref{TABLE} for the special families of \cite{moonen} (see \cite{LMPT2}).

	\subsubsection{Adding slopes 0 and 1}

By implementing \Cref{thm_muord} recursively, 
we obtain a method to increase the genus and the 
multiplicity of the slopes $0$ and $1$ in the Newton polygon by the same amount.
Because of this, in later results
we will aim to minimize the multiplicity of $\{0,1\}$
in the Newton polygon.

\begin{corollary}\label{inf-ord}
Let $\gamma=(m,N,a)$ be a monodromy datum.
	Assume that $Z^\circ(\gamma)[u]$ is non-empty. 
	Then for any $n$ in the semi-group generated by $\{m-t \colon \ t\mid m\}$,
	there exists a $\mu_m$-cover $C \to {\mathbb P}^1$ over $\overline{\mathbb{F}}_p$ 
	where $C$ is a smooth curve with Newton polygon $u\oplus \ord^n$.
\end{corollary}

\begin{proof}
	For $c \in \ZZ_{\geq 1}$ with $c \leq m-1$, let 
	$t={\rm gcd}(m,c)$ and consider the monodromy datum $\gamma_1=(m/t, 3, (c/t, (m-c)/t,0))$. Note  $g_1(\tau)=0$, and $\cf_1(\tau)=0$, for all $\tau\in \CT$. 
	
	Let $\gamma_2=(m,N+1,a')$, where $a'(1)=0$, and $a'(i)=a(i-1)$, for $i=2, \dots ,N+1$. 
	By construction, the pair $\gamma_1,\gamma_2$ is admissible and balanced.
	By \Cref{plus1}, $Z^\circ(\gamma_2)[u]$ is non-empty.
	Set $a_3= (c, m-c, a(1), \ldots, a(N))$; then $\gamma_3=(m, N+2, a_3)$ 
	is the monodromy datum from \Cref{Dmd3} for the pair $\gamma_1,\gamma_2$. 
	By \eqref{Dep}, $\epsilon = m-t$.
	By \Cref{thm_muord}, $Z^\circ_3[u_3]$ is non-empty, where $u_3=u \oplus \ord^{m-t}$.	
 The statement follows by iterating this construction, letting $c$ vary. 
	\end{proof}

\subsubsection{Single Induction} 

We consider inductive systems generated by a single monodromy datum $\gamma_1$.
The next result follows from the observation that if 
the pair $\gamma_1, \gamma_1$ is admissible and if $Z^\circ_1[u_1]$ is non-empty, 
then all hypotheses of \Cref{thm_muord} are satisfied, and continue to be after iterations.

\begin{corollary}\label{infinite-ord}
	Assume there exist $1\leq i< j\leq N$ such that $a(i)+a(j)\equiv 0 \bmod m$.\footnote{This condition implies that $\overline{Z}(\gamma)$ intersects the boundary component $\Delta_0$ of ${\mathcal M}_g$.} Let $r=\gcd(a(i),m)$.
	If $Z^\circ(\gamma)[u]$ is non-empty,
	then there exists a smooth curve over $\overline{\mathbb{F}}_p$ 
	with Newton polygon $u^n\oplus \ord^{(n-1)(r-1)}$, for any $n\in \ZZ_{\geq 1}$.
\end{corollary}

\begin{proof}
	After reordering the branch points, we can suppose that $i=1$, $j=N$.
	We define a sequence of families $Z^{\times n}$ as follows:
	let $Z^{\times 1}=Z$; for $n\geq 2$, let $Z^{\times n}$ be the family constructed from the monodromy datum produced by applying \Cref{Dmd3} to the monodromy data of $Z^{\times (n-1)}$ and $Z^{\times 1}$. 
	For $n \in \ZZ_{\geq 1}$, the pair of monodromy data for $Z^{\times n}$ and $Z^{\times 1}$ satisfies hypotheses (A) and (B).
	Then $u_n:=u^n\oplus \ord^{(n-1)(r-1)}$ is the $\mu$-ordinary Newton polygon for $Z^{\times n}$.
	The statement follows by applying \Cref{thm_muord} repeatedly. 
\end{proof}

The first hypothesis of \Cref{infinite-ord} appears restrictive.
However, from any monodromy datum $\gamma$ with $Z^\circ(\gamma)[u]$ non-empty, 
we can produce a new monodromy datum which satisfies this hypothesis by clutching with
a $\mu_m$-cover branched at only two points.
As a result, \Cref{infinite-ord} can be generalized to \Cref{inf-muord} which holds in much greater generality, 
at the expense of making the defect slightly larger.

\begin{corollary}\label{inf-muord}
	Assume that $Z^{\circ}(\gamma)[u]$ is non-empty and let $t$ be a positive divisor of $m$. Then there exists a smooth curve over $\overline{\mathbb{F}}_p$ with Newton polygon $u^n\oplus \ord^{mn-n-t+1}$, for any $n\in \ZZ_{\geq 1}$. 
\end{corollary}

\begin{proof} 
	For $t=m$, consider the family $Z^{\times 1}$ with monodromy datum $(m,N+2,a')$, where $a'(i)=a(i)$ for $1\leq i\leq N$, $a'(N+1)=a'(N+2)=0$. Then, the statement follows from \Cref{infinite-ord} applied to $Z^{\times 1}$ for $r=m$. 
	
	For $t<m$, consider the family $Z^{\times 1}$ with monodromy datum 
	$(m,N+2, a')$, where $a'(i)=a(i)$ for $1\leq i\leq N$, $a'(N+1)=t$ and $a'(N+2)=m-t$. 
	The $\mu$-ordinary polygon $u'$ of the associated Shimura variety  is $u\oplus \ord^{m-t}$. 
	By \Cref{inf-ord}, $(Z^{\times 1})^\circ[u']$ is non empty.
	Then, the statement follows from \Cref{infinite-ord} applied to $Z^{\times 1}$ for $r=t$,
	by observing that
	$(u')^n\oplus \ord^{(n-1)(t-1)}=u^n\oplus\ord^{mn-n-t+1}$.
\end{proof}

\subsubsection{Double Induction} 

We next consider inductive systems constructed from a pair of monodromy data 
satisfying hypotheses (A) and (B). 
For clarity, we state \Cref{twofamilyinfinite} under the simplifying assumption $m_1=m_2$. 
 \Cref{corolast} contains an example of this result; it also applies to
the pair of monodromy data in the proof of \Cref{table2}. 

\begin{corollary}\label{twofamilyinfinite}
	Let $\gamma_1$ and $\gamma_2$ be a pair of monodromy data with $m_1=m_2$
	satisfying hypotheses (A) and (B). Let $r = \gcd(m,a_1(N_1))$ and recall \Cref{Nshimclutch}. 
	Assume that $Z_1^\circ[u_1]$ and $Z_2^\circ[u_2]$ are both non-empty.
	Then there exists a smooth curve over $\overline{\mathbb{F}}_p$ with Newton polygon 
	$u_1^{n_1} \oplus  u_2^{n_2} \oplus \ord^{(n_1+ n_2 -2)(m-1) + (r-1)}$ for any 
	$n_1,n_2 \in \mathbb{Z}_{\ge 1}$.
\end{corollary}

\begin{proof}
	We apply \Cref{inf-muord} to the family $Z_1$ (resp.\ $Z_2$) with $t=m$.
	The result is a family $Z_1^{\times n_1}$ (resp.\ $Z_2^{\times n_2}$) of smooth curves with Newton polygon 
	$u_1^{n_1}\oplus \ord^{(n_1-1)(m-1)}$ (resp.\ $u_2^{n_2}\oplus \ord^{(n_2-1)(m-1)}$). 
	Since $Z_1$ and $Z_2$ satisfy hypotheses (A) and (B), due to the construction in the proof of \Cref{inf-muord}, $Z_1^{\times n_1}$ and $Z_2^{\times n_2}$ also satisfy hypothesizes (A) and (B).
	The result then follows from \Cref{thm_muord}.
\end{proof}


\section{Hypothesis (B) and the $\mu$-ordinary Newton polygon}\label{sec_muord}

In this section, we prove \Cref{balanced}, 
namely that hypothesis (B) for a pair of signatures $(\cf_1,\cf_2)$ is  a necessary and sufficient condition for the associated Shimura variety $\Sh_1\times\Sh_2$ to intersect the $\mu$-ordinary Newton polygon stratum of $\Sh^*_3$.
Recall that $\sigma$ is Frobenius and $\CO$ is the set of orbits of $\sigma$ in $\CT$.	

\begin{notation}	
For a $\sigma$-orbit $\co \in \CO$,
let $\p_\co$ be the prime above $p$ associated with $\co$ and let $|\co|$ be the size of the orbit.
For each $\tau \in \co$, the order of $\tau$ in $\ZZ/m\ZZ$ is constant and denoted $e_\co$; by definition, $e_\co \mid m$.
Let $\QM_{\p_\co}$ denote the local field which is the completion of $K_{e_\co}$ along the prime $\p_\co$.
\end{notation}
	
	Let  $\CX$ denote the universal abelian scheme over $\Sh=\Sh(H,\hh)$\footnote{
		or more generally the universal semi-abelian variety over its toroidal compactification,}, 
	and $\CX[p^\infty]$ the associated $p$-divisible group scheme.
	Let $x\in\Sh(\overline{\FF}_p)$ and consider the abelian variety $X:=\CX_x$. 
	Let $\nu=\nu(X)$ be the Newton polygon of $X$. 
	We omit the proof of the following. 
	
	\begin{lemma}\label{NPco}
	The $\QM$-action of $\CX$ induces a $\QM\otimes_\QQ\QQ_p$-action on $\CX[p^\infty]$. 
	Thus it induces canonical decompositions
	\[\CX[p^\infty]=\bigoplus_{\co\in\CO}\CX[\p^\infty_\co] \text{ and } \nu=\bigoplus_{\co\in\CO} \nu(\co),\]
	where, for each $\co\in\CO$, the group scheme $\CX[\p^\infty_\co]$ is a $p$-divisible $\QM_{\p_\co}$-module and
	$\nu(\co)=\nu(X[\p^\infty_\co])$ is its Newton polygon. 
	\end{lemma}
		
	For each $\nu \in B=B(H_{\QQ_p},\mu_\hh)$, we write $\nu(\co)$ for its $\co$-component, hence $\nu=\bigoplus_{\co\in\CO} \nu(\co)$. 
	For all $\nu,\nu'\in B$, note that $\nu\leq \nu'$ if and only if $\nu(\co)\leq \nu'(\co)$, for all $\co\in\CO$.

The following lemma says that to prove \Cref{balanced}, it suffices to consider each $\sigma$-orbit $\co$ in $\CT=\ZZ/m_3\ZZ$ separately. 

\begin{lemma}\label{r-1}
	With the same notation and assumption as in \Cref{balanced}: 
		the equality $u_3=u_1^{d_1}\oplus  u_2^{d_2}\oplus \ord^\epsilon$ is equivalent to the system of equalities, 
		for every orbit $\co$ in $\CT$,
		\begin{equation} \label{ENPunramorbit}
		u_3(\co)=u_1^{d_1}(\co)\oplus  u_2^{d_2}(\co)\oplus  \ord^{\epsilon_\co}, 
		\end{equation}
		where $\epsilon_\co:=|\co|$ if
		$e_\co$ is a divisor of $d_1d_2r_0$ but not a divisor of $d_1$ or $d_2$,
		and $\epsilon_\co:=0$ otherwise.
\end{lemma} 

\begin{proof}
	Note that $\epsilon=d_1d_2r_0-d_1-d_2+1$ is equal to the number of $\tau\in\CT$ whose order is a divisor of $d_1d_2r_0$ but not a divisor of $d_1$ or $d_2$; recall $\gcd(d_1,d_2)=1$. 
	Thus $\epsilon=\sum_{\co\in\CO} \epsilon_\co$, and the statement follows from \Cref{NPco} and the discussion below the lemma. 
\end{proof}

\begin{proof}[Proof of \Cref{chain2}]
Fix an orbit $\co$ in $\CT$.
	For $i=1,2$, let $u_i^\dagger$ denote the $\mu$-ordinary Newton polygon of $\Sh(\mu_{m_3},\cf_i^\dagger)$.
By definition,  $u_i^\dagger=u_i^{d_i}$, and  $u_i^\dagger(\co)=u_i^{d_i}(\co)$. 	
	That is, the Newton polygons $u_i^\dagger(\co)$ and $u_i(\co)$ have the same slopes, with 
	the multiplicity of each slope in $u_i^\dagger(\co)$ being $d_i$ times its multiplicity in $u_i(\co)$. 
	Recall the formulas for the slopes and multiplicities of $u(\co)$ from \Cref{Sformulau}.
	
	\smallskip

	\paragraph{{\bf Reduction to a combinatorial problem.}}
	
	The formulas for the slopes and multiplicities
	rely only on the signature type $\cf$ viewed as a $\NN$-valued function on $\CT$, and do not require
	$\cf$ to be a signature associated with a Shimura variety. 
	We regard each signature type $\cf$ as an $\NN$-valued function on $\co$, 
	and denote by $u(\co)$ the Newton polygon defined by the data of slopes in \eqref{eq_slopes}
	and multiplicities in \eqref{eq_multi}.\footnote{the integer $g(\co)$ will be specified in each part of the proof.}  We use subscripts and superscripts to identify 
	various $\NN$-valued functions and their Newton polygons, for example, 
	$\cf_1^\dagger$ and $u_1^\dagger(\co)$.
	
	\Cref{chain2} follows from the claim below taking $R=r_0$, using \Cref{r-1}. 
	
	\smallskip
	
	\paragraph{{\bf Claim}} {\em Let $R$ be a positive integer which divides $\gcd(m_1,m_2)$. 
	Set
		$\delta:=\delta_{d,dR}+\delta_{d_1,d_1d_2}-\delta_{1,d_2}$, with 
		$d:=m_3/\gcd(m_1,m_2)=d_1d_2$ and notations as in \Cref{delta}; set
		$\epsilon_\co:=|\co|$ if $e_\co$ divides  $dR$  but not $d_1$ or $d_2$, and $\epsilon_\co:=0$ otherwise.
				
		Define $\cf_3:=\cf_1^\dagger+\cf^\dagger_2+\delta$.
		Then, the equality \begin{equation}\label{eq_uco}
		u_3(\co)=u^\dagger_1(\co) \oplus u^\dagger_2(\co)\oplus\ord^{\epsilon_\co}
		\end{equation}
		holds if and only if the pair $\cf_1,\cf_2$ satisfies hypothesis (B). }
	
	\smallskip

	\paragraph{{\bf Reduction of claim to the case $d=R=1$}}
	
	We first prove that if $\cf_3=\cf_1^\dagger+\delta$ then
	\begin{equation}\label{eq_uco2}
	u_3(\co)=u^\dagger_1(\co) \oplus \ord^{\epsilon_\co}.
	\end{equation}

Note that the function $\delta(\tau)$ is constant on $\co$, with value $\delta(\co)$ equal to $1$ if $e_\co$ divides  $dR$  but not $d_1$ or $d_2$, and to $0$ otherwise. In particular, $\epsilon_\co=\sum_{\tau\in\co}\delta(\tau)$, $\delta(\co)=\delta(\co^*)$, and by \Cref{Rdual}, $g_3(\co)=g^\dagger_1(\co)+ 2\delta(\co)$.
	
	If $\delta(\co)=0$, then $g_3(\co)=g^\dagger_1(\co)$, and $\cf_3(\tau)=\cf_1^\dagger(\tau)$ for all $\tau\in\co$. Hence, $u_3(\co)=u_1^\dagger(\co)$ which agrees with equality \eqref{eq_uco2} for $\epsilon_\co=0$.
	
	If $\delta(\co)=1$, then $g_3(\co)=g_1^\dagger(\co)+2$, and  $\cf_3(\tau)=\cf_1^\dagger(\tau)+1$ for all $\tau\in\co$.  In particular, $g_3(\co)>\cf_3(\tau) \geq 1$, for all $\tau\in\co$.
	By \eqref{eq_slopes} and \eqref{eq_multi}, both $0$ and $1$ occur as slopes of $u_3(\co)$, 
	with multiplicity respectively $\rho_3(0)=\rho_1(0)+|\co|$ and $\rho_3(1)=\rho_1(1)+|\co|$. 	Each of the slopes $\lambda$ of $u_1^\dagger(\co)$, with $\lambda\not= 0,1$, also occurs for $u_3(\co)$, with  multiplicity $\rho_3(\lambda)=\rho_1(\lambda)$. We deduce that
	$u_3(\co)=u^\dagger_1(\co) \oplus \ord^{|\co|}$, which agrees with equality \eqref{eq_uco2} for $\epsilon_\co=|\co|$.
	
	\smallskip
	
	Equality \eqref{eq_uco2}  is equivalent to the claim for $\cf^\dagger_2(\tau)=0$, for all $\tau\in\co$. 
	Indeed,  hypothesis (B) for the pair $(\cf^\dagger_1,0)$ holds trivially, and equality \eqref{eq_uco} for  $\cf_3=\cf^\dagger_1+\delta$ 	specializes to 
	\eqref{eq_uco2}.

By \eqref{eq_uco2}, we deduce that equality \eqref{eq_uco} holds if and only if it holds in the special case of $d=R=1$, that is if
 $u_3(\co)=u^\dagger_1(\co) + u^\dagger_2(\co)$ when $\cf_3=\cf^\dagger_1+\cf_2^\dagger$.
	By definition, hypothesis (B) holds for the pair $(\cf_1^\dagger, \cf^\dagger_2)$ if and only if it holds for $(\cf^\dagger_1+\delta, \cf^\dagger_2)$.
	Hence, without loss of generality, we may assume $d=R=1$ and $\cf_3=\cf_1+\cf_2$.
	In this case, $g_3(\co)=g_1(\co)+g_2(\co)$ by \Cref{Rdual} and the claim reduces to the following:

	\smallskip
	\paragraph{{\bf Specialized claim:}} {\em Assume $\cf_3=\cf_1 + \cf_2$.
		Then, the equality 
		\begin{equation}\label{eq_ucoS}
	u_3(\co)=u_1(\co)\oplus u_2(\co).
	\end{equation}
		holds if and only if the pair $\cf_1,\cf_2$ satisfies hypothesis (B). }

	\smallskip
	
	\paragraph{{\bf Converse direction: assume hypothesis (B)}}
	We shall prove that the equality \eqref{eq_ucoS} holds, by induction on the integer $s_3+1$, the number of distinct slopes of $u_3(\co)$. 
	More precisely, we shall proceed as follows.  First, we shall establish the base case of induction, for $s_3=0$. Next, we shall prove the equality of multiplicities
	\begin{equation}\label{eq_multiEQ}
	\rho_3(\lambda)=\rho_1(\lambda)+ \rho_2(\lambda)
	\end{equation} 
	for $\lambda=\lambda_3(0)$ the first (smallest) slope of $u_3(\co)$.
	Then, we shall assume $s_3\geq 1$ and show that the inductive hypothesis and equality \eqref{eq_multiEQ} imply equality \eqref{eq_ucoS}, 
	which will complete the argument. In the induction process, $g_i(\co)$ remains unchanged.

	\smallskip
	
	\paragraph{{\bf Base case:}}
	Assume $s_3=0$.
	Then, for all $\tau\in\co$, either $\cf_3(\tau) = g_3(\co)$ or $\cf_3(\tau)=0$.
	The equalities $\cf_3(\tau)=\cf_1(\tau)+\cf_2(\tau)$ and $g_3(\co)=g_1(\co)+g_2(\co)$ imply that   $\cf_3(\tau)=g_3(\co)$ (resp.\ $\cf_3(\tau)=0$) if and only if $\cf_i(\tau)=g_i(\co)$ (resp.\  $\cf_i(\tau)=0$)  for both $i=1,2$.
	We deduce that both hypothesis (B) and \eqref{eq_ucoS} hold in this situation.

	\smallskip
	
	\paragraph{{\bf Equality \eqref{eq_multiEQ}}}
	For $i=1,2,3$, let  $E_i(\max)$ denote the maximal value of $\cf_i$ on $\co$. By definition, $E_i(\max)$ is equal to either $E_i(0)$ or $E_i(1)$. In the first case, $\lambda_i(0)>0$; in the second case, $\lambda_i(0)=0$.
	We claim that hypothesis (B) implies $E_3(\max)=E_1(\max)+E_2(\max)$.
	First, note that hypothesis (B) implies that, for $\omega,\tau\in\co$: 
	\begin{equation}\label{eq_BB}
	\cf_3(\omega)=\cf_3(\tau) \text{ if and only if } \cf_i(\omega)=\cf_i(\tau) \text{ for both } i=1,2.
	\end{equation}
	
	For $i=1,2,3$, set $S_i=\{\tau\in\co\mid \cf_i(\tau)=E_i(\max)\}$. Then, property \eqref{eq_BB} implies that $S_3=S_1\cap S_2$, which in turn implies 
	$E_3(\max)=E_1(\max)+E_2(\max)$.
	
	We claim that hypothesis (B) implies that $S_3=S_i$ for some $i\in\{1,2\}$.
	Without loss of generality, assume that $S_2$ properly contains $S_3$, and let $\tau_0\in S_2-S_3$. Then $\tau_0\not\in S_1$.
	For any $\omega\in S_1$: $\cf_1(\omega)=E_1(\max)>\cf_1(\tau_0)$. Thus, by hypothesis (B), $\cf_2(\omega)\geq \cf_2(\tau_0)=E_2(\max)$. We deduce that $\cf_2(\omega)=E_2(\max)$, hence $S_1\subseteq S_2$.
	
	By the formulas for slopes \eqref{eq_slopes}, the equality $S_1=S_3$ implies that $\lambda_1(0)=\lambda_3(0)$, and the inclusion $S_1\subseteq S_2$ implies $\lambda_1(0)\leq \lambda_2(0)$ (and the equality holds if and only if $S_2=S_3$).
	
	For $i=1,2,3$, let $E_i({\rm next})$ denote the maximal value of $\cf_i$ on 
	$\co-S_3$. 
	For $i=1,3$, $E_i({\rm next})<E_i(\max)$; for $i=2$, $E_2({\rm next})\leq E_2(\max)$ and the equality holds if and only if $S_2$ properly contains $S_3$.
	
	As before by property \eqref{eq_BB}, we deduce that hypothesis (B) implies $E_3({\rm next})=E_1({\rm next})+E_2({\rm next})$.
	By the formulas for multiplicities \eqref{eq_multi}, the two identities, $E_3(\max)=E_1(\max)+E_2(\max)$ and $E_3({\rm next})=E_1({\rm next})+E_2({\rm next})$, imply the desired equality \eqref{eq_multiEQ}.
	
	\smallskip
	
	\paragraph{{\bf Assume $s_3\geq 1$}}
	Then the polygon $u_3(\co)$ has at least two distinct slopes. Our plan is to introduce auxiliary functions $\tilde{\cf}_1(\tau), \tilde{\cf}_2(\tau)$ such that polygon $\tilde{u}_3(\co)$ associated with the function $\tilde{\cf}_3(\tau):=\tilde{\cf}_1(\tau)+\tilde{\cf}_2(\tau)$ has  $s_3$ distinct slopes. 
	
	For $i=1,2$, define $\tilde{\cf}_i(\tau):=\cf_i(\tau)$ for all $\tau\not\in S_3$ and $\tilde{\cf}_i(\tau):=E_i({\rm next})$ for $\tau\in S_3$.
	Note that $\tilde{\cf}_2=\cf_2$ unless $S_2=S_3$.
	By definition, for $i=1,2,3$, $\tilde{E}_i(\max)=E_i({\rm next})$.
	
	For $i=1,3$, and for $i=2$ if $S_2=S_3$, the polygon $\tilde{u}_i(\co)$ shares the same slopes as $u_i(\co)$ except $\lambda_i(0)$ which no longer occurs. For each $t=2, \dots, s_i$, the slope $\lambda_i(t)$ occurs in 
	$\tilde{u}_i(\co)$ with multiplicity equal to $\rho_i(\lambda_i(t))$; while the slope $\lambda_i(1)$ occurs in $\tilde{u}_i(\co)$ with multiplicity
	$\rho_i(\lambda_i(0))+\rho_i(\lambda_i(1))$.
	For $i=2$, if $S_2\not= S_3$, then $\tilde{u}_2(\co)=u_2(\co)$.
	
	Note that $\tilde{u}_3$ has exactly $s_3$ slopes. Hence, by the inductive hypothesis, we deduce that $\tilde{u}_3(\co)=\tilde{u}_1(\co)\oplus \tilde{u}_2(\co)$. This identity, together with \eqref{eq_multiEQ} and the above computation of $\tilde{\rho}_i(\lambda_i(1))$, implies that $u_3(\co)=u_1(\co) \oplus u_2(\co)$.
	
	\smallskip
	
	\paragraph{{\bf Forward direction: assume \eqref{eq_ucoS}}}
	We shall prove that the pair $(\cf_1,\cf_2)$ satisfies hypothesis~(B), arguing by contradiction. Supposing hypothesis (B) does not hold, we shall define auxiliary functions $\bar{\cf}_1,\bar{\cf}_2$ obtained by precomposing $\cf_1,\cf_2$  with a permutation of $\co$, such that the Newton polygon $\bar{u}_3(\co)$ associated with the function $\bar{\cf}_3(\tau):=\bar{\cf}_1(\tau)+\bar{\cf}_2(\tau)$ is strictly above  $u_3(\co)$, i.e., 
	$\bar{u}_3(\co)<u_3(\co)$.   
	By the formulas for slopes and multiplicities \eqref{eq_slopes} and \eqref{eq_multi}, we see that precomposing a function $\cf$ with a permutation of $\co$ does not change the associated polygon $u(\co)$. Hence, for $i=1,2$ we deduce $\bar{u}_i(\co)=u_i(\co)$.
	On the other hand, by repeating the permutation process, we eventually end up with a pair $(\bar{\cf}_1,\bar{\cf}_2)$ which satisfies hypothesis (B) and by the above argument, hypothesis (B) implies that $\bar{u}_1(\co)\oplus \bar{u}_2(\co)=\bar{u}_3(\co)$. Hence $\bar{u}_3(\co)=u_1(\co)\oplus u_2(\co)$, which equals $u_3(\co)$ by \eqref{eq_uco} and contradicts the conclusion that $\bar{u}_3(\co)<u_3(\co)$.   
	
	\smallskip
	
	\paragraph{{\bf Contradict hypothesis (B)}}
	Thus, there exist $\omega_0,\eta_0\in \co$ such that
	\begin{equation*}
	\cf_2(\omega_0)>\cf_2(\eta_0) \text{ and }\cf_1(\omega_0)<\cf_1(\eta_0). 
	\end{equation*}
	Let $\gamma$ denote the permutation of $\co$ which switches $\omega_0$ and $\eta_0$, and define $\bar{\cf}_1(\tau):=\cf_1(\gamma(\tau))$ and $\bar{\cf}_2(\tau):=\cf_2(\tau)$. Set $\bar{\cf}_3(\tau):=\bar{\cf}_1(\tau)+\bar{\cf}_2(\tau)$.
	Then, $\bar{\cf}_3(\tau)=\cf_3(\tau)$ except for $\tau=\omega_0,\eta_0$. Note that $\bar{\cf}_3(\omega_0)>\cf_3(\omega_0)$,  $\bar{\cf}_3(\eta_0)<\cf_3(\eta_0)$,
	and $\bar{\cf}_3 (\omega_0)>\cf_3(\eta_0)$ (also, $\bar{\cf}_3 (\eta_0)<\cf_3(\omega_0)$).
	
	We claim that $\bar{u}_3(\co)<u_3(\co)$, meaning that  $\bar{u}_3(\co)$ and $u_3(\co)$ share the same endpoints (this follows from the equality $\bar{g}_3(\co)=g_3(\co)$), and that $\bar{u}_3(\co)$ lies strictly above $u_3(\co)$.
	
	We first show that possibly after sharing the first several slopes, $\bar{u}_3(\co)$ admits a slope which is strictly larger than the corresponding one in $u_3(\co)$.
	Let us consider the value $A=\bar{\cf}_3(\omega_0)$. Note that $A>\bar{\cf}_3(\omega_0)\geq 0$.
	If  $A=\cf_3(\tau)$ for some $\tau\in\co$, 
	say $A=E_3(t)$ for some $t\in\{0,\dots, s_3\}$. Then
	by the formulas for slopes and multiplicities \eqref{eq_slopes} and \eqref{eq_multi}, we deduce that  the first $t$ slopes, and their multiplicities, of $\bar{u}_3(\co)$ and $u_3(\co)$ agree, but the $(t+1)$-st slope of $\bar{u}_3(\co)$ is strictly larger that the $(t+1)$-st slope of $u_3(\co)$. 
	If $\cf_3(\tau)\not= A$ for all $\tau\in\co$, let $t \in\{0,\dots, s_3\}$ be such that $E_3(t)> A> E_3(t+1)$. 
	We deduce that  the first $t$ slopes of $\bar{u}_3(\co)$ and $u_3(\co)$ agree, and so do the multiplicities of the first $t-1$ slopes, but 
	the multiplicity of the $t$-th slope of $\bar{u}_3$ is strictly smaller than that of $u_3(\co)$.  
	
	By similar arguments for the subsequent slopes and multiplicities, $\bar{u}_3(\co)$ never drops strictly below $u_3(\co)$, but it might (and often does) agree with $u_3(\co)$ for large slopes.
\end{proof}


\section{The Torelli locus and the non $\mu$-ordinary locus of Shimura varieties}\label{sec_nonord}

In this section, we study the intersection of the open Torelli locus with Newton polygon strata
which are not $\mu$-ordinary in PEL-type Shimura varieties.
The main result, Theorem~\ref{differentm}, provides a method 
to leverage information from smaller dimension to larger dimension.
This theorem is significantly more difficult than \Cref{thm_muord}; we add an extra compatibility condition
to maintain control over the codimensions of the Newton polygon strata.
This is the first systematic result on this topic that we are aware of.

For applications, we find situations in which Theorem \ref{differentm}
can be applied infinitely many times; from this, we produce systems of 
infinitely many PEL-type Shimura varieties for which we can verify that the 
open Torelli locus intersects certain non $\mu$-ordinary Newton polygon strata.
See \Cref{Cadd01non,inf-nu,twofamilynext} and \Cref{Applications1} for details.

\begin{notation}\label{Dep1}
Let $\gamma_1=(m_1,N_1,a_1)$,  $\gamma_2=(m_2,N_2,a_2)$ be an ordered pair of 
(generalized) monodromy data which satisfies hypothesis (A). Assume that $m_1|m_2$.
Set $d:=m_2/m_1$ and $r:=\gcd(m_1, a_1(N_1))$.  
Then, \eqref{Dep} specializes to 
$\epsilon=d(r-1)$ and $g_3=dg_1+g_2+\epsilon$.
In particular, $\epsilon=0$ if and only if $r=1$. 
\end{notation}

\Cref{why} explains why we restrict to the case $m_1|m_2$.

\subsection{Hypothesis (C)}

To study the Newton polygons beyond the $\mu$-ordinary case, we introduce an extra hypothesis. 

\begin{definition}
	Given a Newton polygon $\nu$, the \emph{first slope} $\lambda_{1st}(\nu)$ is the smallest slope of $\nu$ and the \emph{last slope} $\lambda_{last}(\nu)$ is the largest slope of $\nu$.
	If $\nu$ is symmetric with $q$ distinct slopes, the \emph{middle slope} 
	$\lambda_{mid}(\nu)$ is the $\lfloor\frac{q+1}{2}\rfloor$-st slope of $\nu$.
\end{definition}

\begin{definition}\label{def_hp3}
	An ordered pair of monodromy data $\gamma_1, \gamma_2$ is \emph{compatible} if
	the slopes of the $\mu$-ordinary Newton polygons $u_1$ and $u_2$ satisfy:
	
hypothesis (C): for each orbit
	 $\co\in\CO$, every slope of $u_1(\co)$ is in the range 
	 \[[0,1]\setminus(\lambda_{1st}(u_2(\co)), \lambda_{last}(u_2(\co))).\]
	By convention, the condition in the previous line holds for $\co$ 
	if $u_i(\co)$ is empty for either $i=1,2$.
	If a pair of monodromy data is compatible, then we
	write $u_1\ll_{\text{(C)}} u_2$. 
\end{definition}

\begin{remark}\label{trivialC}
	\begin{enumerate}
		\item
		Hypothesis (C) holds for $\co$ if either 
		$u_1(\co)$ has slopes in $\{0,1\}$ or
		$u_2(\co)$ is supersingular.  In particular, if $u$ is ordinary or supersingular, then $u\ll_{\text{(C)}} u$.
		\item
		If $\co=\co^*$, then hypothesis (C) holds for $\co$ if and only if $\lambda_{mid}(u_1(\co))\leq \lambda_{1st}(u_2(\co))$. 
		\item If $u_1\ll_{\text{(C)}} u_2$ then $u_1^d \ll_{\text{(C)}} u_2$ for all $d \in \ZZ_{\geq 1}$. 
	\end{enumerate}
\end{remark}

\begin{remark}\label{rmk_hp3}
Let $\gamma_1,\gamma_2$ be a pair of monodromy data as in \Cref{Dep1}. By \Cref{Sformulau}, 
\[\lambda_{1st} (u_2(\co))=\frac{1}{|\co|}\#\{\tau\in\co \mid \cf_2(\tau)=g_2(\co)\} \text{ and } \lambda_{last} (u_2(\co))=\frac{1}{|\co|}\#\{\tau\in\co \mid \cf_2(\tau)> 0\}.\]
Hypothesis (C) holds for $\co$ if and only if there exists an integer $E(\co)\in[0,g_1(\co)]$ such that
\[\#\{\tau\in\co \mid \cf_1^\dagger(\tau)>E(\co)\} \leq \#\{\tau\in\co \mid \cf_2(\tau)=g_2(\co)\}, \text{ and}\] 
\[\#\{\tau\in\co \mid \cf_1^\dagger(\tau)\geq E(\co)\}\geq \#\{\tau\in\co \mid \cf_2(\tau)>0 \}.\]
\end{remark}
	
The next statement  follows  from \Cref{def_hp3} and \Cref{rmk_hp3}.

\begin{lemma}\label{ullu}
	The following are equivalent: 
				$u\ll_{\text{(C)}} u$;  for each $\co\in\CO$, the Newton polygon
	 $u(\co)$ has at most two distinct slopes; and, for each $\co\in\CO$,
	 there exists an integer $E(\co)\in[0,g(\co)]$ such that $\cf(\tau)\in \{0,E(\co),g(\co)\}$ for all $\tau\in\co$.
				\end{lemma}

Unlike hypothesis (B), hypothesis (C) 
	does not behave well under induction in general. 
	\Cref{HP3induce} identifies two instances when it does. We omit the proof.

\begin{lemma}\label{HP3induce}
		\begin{enumerate}
		\item If $u_1\ll_{\text{(C)}} u_2$ then $u^n_1\oplus \ord^l\ll_{\text{(C)}} u_2$,  for any $n,l\in\ZZ_{\geq 1}$.	
		
		\item 
		If $u_1\ll_{\text{(C)}} u_2$ and $u_2\ll_{\text{(C)}} u_2$, then $u^n_1\oplus u^m_2\oplus\ord^l\ll_{\text{(C)}} u_2$,  for any $n,m,l\in\ZZ_{\geq 1}$.	
	\end{enumerate}
\end{lemma}

\subsection{The significance of hypothesis (C)}

Hypothesis (C) is sufficient to prove the geometric condition on the Newton polygon stratification in \Cref{codimension1} below. We use hypothesis (C) to prove the surjectivity of the map in \eqref{EmaponKottwitz}.

\begin{proposition}\label{codim}\label{codimension1}
Let $\gamma_1=(m_1,N_1,a_1),\gamma_2=(m_2,N_2,a_2)$ be an ordered pair of monodromy data as in \Cref{Dep1}. 
Assume the pair satisfies hypotheses (A), (B), and (C). 
Consider the monodromy datum $\gamma_3 $ as in \Cref{Dmd3}.  
	Then, for any Newton polygon $\nu_2\in B_2$,
	\begin{equation} \label{Ecodimequal}
\codim (\Sh_2[\nu_2],\Sh_2)=\codim (\overline{\Sh}_3[u_1^d\oplus  \nu_2 \oplus \ord^{\epsilon}],\overline{\Sh}_3).
\end{equation}
\end{proposition}

The following lemma is a reformulation of \Cref{balanced}.

\begin{lemma}
	  Hypothesis (B) is equivalent to the assumption that \eqref{Ecodimequal} holds for $\nu_2=u_2$. 
\end{lemma}
In fact, the proof below shows that if hypothesis (B) does not hold, then \eqref{Ecodimequal} is false for all $\nu_2\in B_2$. 

	\begin{proof}[Proof of \Cref{codimension1}] 
For any Kottwitz set $B$ and any $\nu\in B$, let $B(\nu)=\{t \in B\mid t \geq \nu\}$.

We first prove the case when $\epsilon=0$. 
Consider the map 
$\Sigma:B_2 \to B_3,$ where $t\mapsto u_1^d\oplus  t$.

 We first note that $\Sigma$ is an order-preserving injection.\footnote{For convenience, we use hypothesis (C) to construct the order-preserving map \eqref{EmaponKottwitz}; however, this part can be proved without using this hypothesis.}
  Let $t\in B_2$. For each orbit $\co$  in $\CO$,
let $q_1=q_1(\co)$ (resp. $q_1'=q_1'(\co)$) be the number of distinct slopes of 
$u_1(\co)$ in $[0, \lambda_{1st}(u_2(\co)]$ (resp. $[\lambda_{last}(u_2(\co)), 1]$).
By hypothesis~(C), for the Newton polygon $u_1^d\oplus t$, the first $q_1$ and the last $q_1'$ slopes of $u_1(\co)^d\oplus t(\co)$ are the slopes of $u_1(\co)^d$ with the same multiplicities and the rest are the slopes of $t(\co)$ with the same multiplicities.
So, if $t\leq t'$ in $B_2$, then 
$u_1^d \oplus  t\leq u_1^d \oplus  t'$ in $B_3$.
In particular, the map $\Sigma$ induces an injection on the ordered sets
\begin{equation} \label{EmaponKottwitz}
B_2(\nu_2) \to B_3(u_1^d\oplus  \nu_2), \ t\mapsto u_1^d\oplus  t.
\end{equation}

By \eqref{eqn_codim}, to conclude, it suffices to prove that under
hypotheses~(B) and (C) the map in \eqref{EmaponKottwitz} is also surjective. 
By \Cref{balanced}, hypothesis~(B) implies that $u_3=u_1^d\oplus u_2$.
Hence, for a Newton polygon $v\in B_3(u_1^d\oplus  \nu_2)$, then
$u_1^d\oplus\nu_2\leq v\leq u_3=u_1^d\oplus u_2$.
By the paragraph after \Cref{NPco}, 
\begin{eqnarray*}
u_1(\co)^d\oplus \nu_2(\co)  \leq v(\co)  \leq 
u_1(\co)^d\oplus u_2(\co).
\end{eqnarray*}

By hypothesis (C), the inequalities above imply that $v(\co)$ and $u_1(\co)^d$
share the first $q_1$ and last $q_1'$ slopes, with the same multiplicities
except for the $q_1$-th slope (resp. $q_2$-th slope) which may occur with higher multiplicity in the former if $u_1(\co)$ has slope $\lambda_{1st}(u_2(\co))$ (resp. $\lambda_{last}(u_2(\co))$).
We deduce that each $v\in B_3(u_1^d\oplus  \nu_2) $ is of the form $v=u_1^d\oplus t$ for some  $t\in B_2(\nu_2)$; 
thus the map in \eqref{EmaponKottwitz} is surjective.

If $\epsilon \not = 0$, by \eqref{eqn_codim_cpt}, the same argument still applies,  with $u_1^d$ replaced by $u_1^d\oplus \ord^\epsilon$.
\end{proof}

\begin{remark}\label{why}
Let $\gamma$ be a monodromy datum, of signature $\cf$. 
For $d>1$, let $\gamma^\dagger$ be the induced monodromy datum, of signature $\cf^\dagger$, as in \Cref{Nind}. 
The map $(\cdot)^\dagger: B(\cf)\to B(\cf^\dagger)$, $\nu\mapsto \nu^d$, is injective and order-preserving, but is not surjective in general. If $(\cdot)^\dagger$ is not surjective, then, by \eqref{eqn_codim}, there exists $\nu\in B$ such that $\codim (\Sh[\nu],\Sh)\neq \codim (\Sh^\dagger[\nu^d],\Sh^\dagger)$.
For example, if $\cf=(1)$, and $d=2$, this happens when $\nu=(1/2,1/2)$.
\end{remark}

\subsection{The second main result}

The next result also provides a partial positive answer to \Cref{Coortconj} when $\epsilon = 0$.

\begin{theorem}\label{differentm}
	Let $\gamma_1,\gamma_2$ be an ordered pair of monodromy data as in \Cref{Dep1}.
		 Assume it satisfies hypotheses (A), (B), and (C).  Let $\epsilon=d(r-1)$. 
	 Consider the monodromy datum $\gamma_3$ from \Cref{Dmd3}. 	Let $\nu_2 \in B_2$.
	If $Z^{\circ}_1[u_1]$ and $Z^{\circ}_2[\nu_2]$ are non-empty, and  $Z^{\circ}_2[\nu_2]$ contains an irreducible component $\Gamma_2$ such that 
	\begin{equation} \label{Ecodimcond}
	\codim(\Gamma_2, Z_2)=\codim(\Sh_2[\nu_2],\Sh_2),
	\end{equation}
	then $Z^{\circ}_3[ u_1^d\oplus  \nu_2 \oplus \ord^{\epsilon}]$  is non-empty and contains an irreducible component $\Gamma_3$ 	such that 
	\[\codim(\Gamma_3, Z_3)=\codim(\Sh_3[ u_1^d\oplus  \nu_2 \oplus \ord^{\epsilon}],\Sh_3).\]
\end{theorem}

\begin{remark}\label{rmkcodim}
		As seen in \Cref{Sexceptional}, hypothesis (C) is not  a necessary condition and it can occasionally be removed.		Specifically,
		\Cref{differentm} still holds with hypothesis (C) replaced by the weaker (but harder to verify) assumption that \eqref{Ecodimequal} holds for the given non $\mu$-ordinary Newton polygon $\nu_2\in B_2$. 
\end{remark}

\begin{remark}\label{rmk_goodinput}
	If $Z_2$ is one of Moonen's special families from \cite{moonen} and $Z_2^{\circ}[\nu_2]$ is non-empty, 
	then every irreducible component of $Z_2^{\circ}[\nu_2]$ satisfies the codimension condition \eqref{Ecodimcond}.
\end{remark}

\begin{proof}[Proof of \Cref{differentm}]
By \Cref{plus1}, without loss of generality, we may assume that the inertia types $a_1$ and $a_2$ 
contain no zero entries if $r < m_1$ and no zero entries other than $a_1(N_1)=a_2(1)=0$ if $r=m_1$.
Set $l:= \codim(\Sh_2[\nu_2],\Sh_2)$.
Recall the clutching morphism $\kappa: \tilde{Z}_1 \times \tilde{Z}_2 \to \tilde{Z}_3$ 
from \Cref{Pclutchadd}.
Note that
\begin{equation} \label{Edimcompare}
\dim(Z_3)  = N_3-3=(N_1-3)+(N_2-3)+1.
\end{equation}

We distinguish three cases: $\epsilon=0$, $\epsilon\neq 0$ and $r< m_1$, and $\epsilon\neq 0$ and $r=m_1$.

\subsubsection*{Assume $\epsilon=0$}

Then \eqref{Edimcompare} implies that
\begin{equation} \label{Edimcompare2}
\dim(Z_3)  =  \dim(Z_1)+\dim(Z_2) +1=\dim(\tilde{Z}_1) + \dim(\tilde{Z}_2) +1.
\end{equation}
By \Cref{Dtype3},
$g_3=dg_1+g_2$ and the formula for $\cf_3$ is in \eqref{signaturesum}.
The clutching morphism $\kappa$ is compatible with the morphism
$\iota: \Sh_1\times \Sh_2\rightarrow \Sh_3$ given by $\iota(\CX_1,\CX_2)=\CX_1^d\oplus \CX_2$.
Since the map $\tilde{Z}_i\rightarrow \overline{Z}_i$ is finite, $\dim(\tilde{Z}_i[\nu_i])=\dim(\overline{Z}_i[\nu_i])$ for any $\nu_i\in B_i$.

	 By \Cref{codim}, $l=\codim(\Sh_3[u_1^d\oplus  \nu_2],\Sh_3)$. Let $\tilde{\Gamma}_2$ denote the Zariski closure 
of the preimage of $\Gamma_2$ in $\tilde{Z}_2[\nu_2]$. 
Apriori, $\tilde{\Gamma}_2$ may not be irreducible, in which case we replace it
by one of its irreducible components.
Then $\dim(\tilde{\Gamma}_2)=\dim(\Gamma_2)$.
	
	Let $W:=\kappa(\tilde{Z}_1[u_1], \tilde{Z}_2[\nu_2])$.
	Since $W \subseteq \tilde{Z}_3[u_1^d\oplus  \nu_2]$,
	then $\tilde{Z}_3[u_1^d\oplus  \nu_2]$ is non-empty.
	By \eqref{Edimcompare2}, $\kappa(\tilde{Z}_1, \tilde{Z}_2)$ has codimension $1$ in $\tilde{Z}_3$.
	In addition, $W$ 
	is an open and closed substack of the intersection of 
	$\kappa(\tilde{Z}_1, \tilde{Z}_2)$ and $\tilde{Z}_3[u_1^d\oplus\nu_2]$.
	 By \cite[page 614]{V:stack}, every irreducible component of 
	 $W$ 
	 has codimension at most $1$ in the irreducible component of $\tilde{Z}_3[u_1^d\oplus  \nu_2]$ which contains it. 
	 Note that $\kappa(\tilde{Z}_1[u_1], \tilde{\Gamma}_2)$ is an irreducible component of $W$.
	 Let $\tilde{\Gamma}_3$ be the irreducible component of 
	 $\tilde{Z}_3[u_1^d\oplus  \nu_2]$ which contains 
	 $\kappa(\tilde{Z}_1[u_1], \tilde{\Gamma}_2)$. 
	It follows that ${\rm codim}(\kappa(\tilde{Z}_1[u_1], \tilde{\Gamma}_2), \tilde{\Gamma}_3) \leq 1$.
	 So 
	\[\dim(\tilde{\Gamma}_3)=\begin{cases}  \dim(Z_1)+\dim(Z_2)-l &\text{ if $\tilde{\Gamma}_3=\kappa(\tilde{Z}_1[u_1], \tilde{\Gamma}_2)$,} \\ \dim(Z_1)+\dim(Z_2)-l +1 &\text{ otherwise.}
	\end{cases}\]
	
	On the other hand, for all  $b\in B_3$, by \eqref{eqn_codim} and the de Jong--Oort purity theorem \cite[Theorem 4.1]{JO}, the codimension of any irreducible component of $\tilde{Z}_3[b]$ in $\tilde{Z}_3$ is no greater than $\length (b)= \codim(\Sh_3[b],\Sh_3)$. 
	For $b=u_1^d\oplus  \nu_2$, by \eqref{Edimcompare2},
	this yields  \[\dim(\tilde{\Gamma}_3) \geq \dim(Z_3)-l=\dim(Z_1)+\dim(Z_2)+1-l.\]
	We deduce that ${\rm codim}(\tilde{\Gamma}_3, \tilde{Z}_3) = l$
	and that $\tilde{\Gamma}_3$ strictly contains $\kappa(\tilde{Z}_1[u_1], \tilde{\Gamma}_2)$. 
	
	Let $\bar{\Gamma}_3$ denote the image of $\tilde{\Gamma}_3$ via the forgetful map $\tilde{Z}_3\rightarrow \bar{Z}_3$. Define $\Gamma_3=\bar{\Gamma}_3\cap Z_3^{\circ}$. To finish the proof, we only need to show that $\Gamma_3$ is non-empty.
	Therefore, it suffices to show that $\tilde{\Gamma}_3$ is not contained in the image of any other clutching 
	map from \Cref{Pclutchadd}.  Since $r=1$, by \Cref{PNPC3} the points in $W$ represent curves of compact type, 
	thus $\tilde{\Gamma}_3\cap\tilde{Z}^c_3$ is non empty. 
	
	To finish, we argue by contradiction;  
	suppose $\tilde{\Gamma}_3$ is contained in the image of any of the other clutching maps in $\tilde{Z}^c_3$.
This would imply that all points of $\kappa(\tilde{Z}_1[u_1], \tilde{\Gamma}_2)$
represent $\mu_m$-covers of a curve of genus $0$ comprised of at least $3$ projective lines.
This is only possible if all points of either $\tilde{Z}_1[u_1]$ or $\tilde{\Gamma}_2$
represent $\mu_m$-covers of a curve of genus $0$ comprised of at least $2$ projective lines.
This would imply that either $Z^{\circ}_1[u_1]$ or $\Gamma_2 \subset Z_2^\circ[\nu_2]$ is empty, which contradicts the hypotheses of the theorem. 

\subsubsection*{Assume $\epsilon\neq0$ and $r<m_1$}
	By the same argument as when $\epsilon=0$, there exists an irreducible component $\tilde{\Gamma}_3$ of $\tilde{Z}_3[u_1^d \oplus \nu_2 \oplus \ord^{\epsilon}]$ of codimension $l$ such that $\tilde{\Gamma}_3$ strictly contains $\kappa(\tilde{Z}_1[u_1], \tilde{\Gamma}_2)$.  To finish the proof, we only need to show that $\tilde{\Gamma}_3$ is not contained in the boundary of $\tilde{\CM}_{\mu_m}$. 
	As before, $\tilde{\Gamma}_3$ is not contained in the image of any of the 
	clutching maps in $\tilde{Z}^c_3$.
	Suppose that $\tilde{\Gamma}_3$ is contained in the image of any of the clutching maps not in $\tilde{Z}^c_3$.
	By keeping careful track of the toric rank, one can check that 
	this implies that 
	the points of either $\tilde{Z}_1[u_1]$ or $\tilde{\Gamma}_2$ represent $\mu_m$-covers
	of curves that are not of compact type.
	This would imply that either $Z^{\circ}_1[u_1]$ or $\Gamma_2 \subset Z_2^\circ[\nu_2]$ is empty, which contradicts the hypotheses of the theorem.

\subsubsection*{Assume $\epsilon\neq 0$ and $r=m_1$}
By \Cref{plus1}, for $i=1,2$, 
the fibers of the forgetful map 
$f_i:\tilde{Z}^\circ_{i}\rightarrow {Z}_i^\circ$ have pure dimension $1$. 
Let $\tilde{\Gamma}_{2}'$ be an irreducible component of the preimage via $f_2$ of $\Gamma_2$; it is in $\tilde{Z}^\circ_{2}[\nu_2]$. 
Let $\tilde{\Gamma}_{1}'$ be an irreducible component of the
preimage via $f_1$ of $Z^\circ_1[u_1]$; it is in $\tilde{Z}^\circ_{1}[u_1]$.
Then ${\rm dim}(\tilde{\Gamma}_{2}') = {\rm dim}(\Gamma_2)+1$.
Similarly, ${\rm dim}(\tilde{\Gamma}_{1}') = {\rm dim}(Z_1[u_1])+1$.
	
	Let $\tilde{\Gamma}_3$ be the irreducible component of $\tilde{Z}_3[u_1^d\oplus  \nu_2 \oplus \ord^{\epsilon}]$ 
	that contains the image $\kappa(\tilde{\Gamma}'_1, \tilde{\Gamma}'_{2})$. As before, $\dim(\tilde{\Gamma}_3) \geq \dim(\kappa(\tilde{\Gamma}'_1, \tilde{\Gamma}'_{2}))+1$. The rest of the proof follows in the same way as when $r<m_1$, by taking $\Gamma_3=\overline{\Gamma}_3\cap Z^{\circ}_3$, where $\overline{\Gamma}_3$ is the image of $\tilde{\Gamma}_3$ via the forgetful map. To obtain the dimension inequality, note that 
	\[{\rm dim}(\kappa(\tilde{\Gamma}'_1, \tilde{\Gamma}'_{2}))=1+\dim(Z_1[u_1]) + 1 +\dim(\Gamma_2)=2+\dim(Z_1) +\dim(Z_2) - l,\] where $l=\codim (\Sh_2[\nu_2], \Sh_2)$. 
	In this case, ${\rm dim}(Z_i)=N_i-4$ for $i=1,2$.  By \eqref{Edimcompare}, 
	\begin{equation} 
	{\rm dim}(Z_3) = (N_1-4) + (N_2-4) + 3 = {\rm dim}(Z_1) + {\rm dim}(Z_2) +3.
	\end{equation}
	On the other hand, by the de Jong--Oort purity theorem \cite[Theorem 4.1]{JO},  
		\[\dim(\tilde{\Gamma}_3)\geq\dim(Z_3) - l=\dim(Z_1)+\dim(Z_2) + 3-l=
	{\rm dim}(\kappa(\tilde{\Gamma}'_1, \tilde{\Gamma}'_{2}))+1.\qedhere\]
\end{proof}

\subsection{Infinite clutching for non $\mu$-ordinary}\label{sec_infclut_nu}

This section is similar to Section \ref{Sinfclut}, in that 
we find situations in which \Cref{differentm} 
can be implemented recursively, infinitely many times, except that we now focus on non $\mu$-ordinary Newton polygons.

Let $\gamma=(m,N,a)$ be a monodromy datum and let $\nu \in B(\gamma)$.

\begin{corollary}  \label{Cadd01non} (Extension of \Cref{inf-ord})
	Assume $Z^\circ(\gamma)[\nu]$ is non-empty and contains an irreducible component 
	$\Gamma$ such that $\codim(\Gamma, Z(\gamma))=\codim(\Sh[\nu],\Sh)$. 
	Then for any $n$ in the semi-group generated by $\{m-t \colon \ t\mid m\}$,
	there exists a $\mu_m$-cover $C \to {\mathbb P}^1$ over $\overline{\mathbb{F}}_p$ 
	where $C$ is a smooth curve with Newton polygon $\nu \oplus \ord^n$.
		\end{corollary}
		
		\begin{proof}
		Let $\gamma_1$ be as in the proof of \Cref{inf-ord}.  Note that $u_1(\co)$ is empty for all $\co$.
		So the pair $\gamma_1, \gamma$
		satisfies hypothesis (C), in addition to (A) and (B).
		The proof is then the same as for \Cref{inf-ord}, replacing \Cref{thm_muord} with \Cref{differentm}.
		\end{proof}

\begin{corollary}\label{inf-nu}
(Extension of Corollaries \ref{infinite-ord} and \ref{inf-muord})
Let $\epsilon = (n-1)(r-1)$ 
if there exist $1\leq i< j\leq N$ such that $a(i)+a(j)\equiv 0 \bmod m$, and  $\epsilon = (n-1)(m-1)$ otherwise.
Assume $Z^\circ(\gamma)[\nu]$ is non-empty and
	contains an irreducible component $\Gamma$ such that 
		$\codim (\Gamma, Z(\gamma))=\codim (\Sh[\nu], \Sh)$.
		Assume $u \ll_{\text{(C)}}  u$.
	Then for any $n\in \ZZ_{\ge 1}$, there exists a smooth curve with Newton polygon $u^{n-1}\oplus \nu \oplus \ord^\epsilon$.
	\end{corollary}

\begin{proof}
The result is true when $n=1$ by hypothesis.
For $n \geq 2$, we use \Cref{infinite-ord} (resp.\  \Cref{inf-muord} with $t=m$) to construct a family $Z^{\times n-1}$ with Newton polygon $u^{n-1} \oplus \ord^{(n-2)(r-1)}$ (resp.\ $u^{n-1} \oplus \ord^{(n-2)(m-1)}$). 
	The pair of monodromy data of the families $Z^{\times n-1}$ and $Z$ satisfies hypotheses (A) and (B).
	Since $u \ll_{\text{(C)}}  u$, by \Cref{HP3induce} (1), the pair also satisfies hypothesis (C). 
	Hence we conclude by \Cref{differentm}.
\end{proof}

\begin{corollary} \label{twofamilynext}
	With notation and hypotheses as in \Cref{twofamilyinfinite}, 
	assume furthermore that for some $\nu_2 \in B(\gamma_2)$,
$Z^\circ_2[\nu_2]$ is non-empty
and contains an irreducible component $\Gamma$ such that 
$\codim(\Gamma,Z_2)=\codim(\Sh_2[\nu_2],\Sh_2)$.
Also assume that $u_1\ll_{\text{(C)}} u_2$ and $u_2 \ll_{\text{(C)}}  u_2$.
Then there exists a smooth curve with Newton polygon 
$u_1^{n_1} \oplus  u_2^{n_2-1} \oplus  \nu_2 \oplus \ord^{(n_1+n_2-2)(m-1)+ (r-1)}$. 
\end{corollary}

\begin{proof}
If $n_2=1$, we first apply \Cref{inf-muord} with $t=m$ to produce a family 
$Z_3$ with Newton polygon $u_1^{n_1} \oplus \ord^{(n_1-1)(m-1)}$.
Note that $Z_3$ and $Z_2$ satisfy hypotheses (A) and (B).
Since $u_1\ll_{\text{(C)}} u_2$, by \Cref{HP3induce} (1), $Z_3$ and $Z_2$ also satisfy hypothesis (C).
Applying \Cref{differentm} produces a smooth curve with Newton polygon 
$u_1^{n_1} \oplus \nu_2 \oplus \ord^{(n_1-1)(m-1) +(r-1)}$.

For $n_2 \geq 2$, we apply \Cref{twofamilyinfinite} to produce a family
$Z_4$ with Newton polygon $u_1^{n_1} \oplus u_2^{n_2-1} \oplus \ord^{(n_1+n_2-3)(m-1)}$.
Since $u_1\ll_{\text{(C)}} u_2$ and $u_2 \ll_{\text{(C)}}  u_2$, by \Cref{HP3induce} (2),
$Z_4$ and $Z_2$ satisfy hypotheses (A), (B), and (C).
Applying \Cref{Rmodifyslightly}, we produce generalized monodromy data for $Z_4$ and $Z_2$ by marking an additional 
unramified fiber.  In this situation, the toric rank is $\epsilon'=m-1$.
Applying \Cref{differentm} to the generalized families for $Z_4$ and $Z_2$ completes the proof.
\end{proof}

\Cref{twofamilynext} applies to the pair of monodromy data in the proof of \Cref{table2}.

\subsection{An exceptional example} \label{Sexceptional}

We give an example of a 
pair of monodromy data, and non $\mu$-ordinary Newton polygon $\nu_2$,  satisfying hypotheses (A) and (B), but not (C), for which \eqref{Ecodimequal} can be verified directly.
Furthermore, as the Kottwitz set $B_2$ has size 2,  this example also shows that hypothesis (C) is sufficient but not necessary for \Cref{codimension1} to hold.

Recall that $\sss$ is the Newton polygon $(1/2,1/2)$.

\begin{proposition}\label{exception}
	If $p \equiv 7 \bmod 8$ is sufficiently large, then there exists a smooth curve over $\overline{\mathbb{F}}_p$ of genus $9$ with Newton polygon $\sss^7\oplus \ord^2$. 
\end{proposition}

\begin{proof}
		Let $Z_2=Z(8, 4, (4,2,5,5))$. Then $Z_2$ is the special family $M[15]$ in \cite[Table 1]{moonen}, and the associated Shimura variety $\Sh_2$ has signature type $\cf_2=(1,1,0,0,2,0,1)$.  
	At any prime $p\equiv 7 \bmod 8$, the $\mu$-ordinary Newton polygon is 
	$u_2=\ord^2 \oplus \sss^3$ and the basic Newton polygon is $\nu_2=\sss^5$ \cite[Section 6.2]{LMPT2}.  
	
	Let $Z_1=Z(4,3,(1,1,2))$, which has signature $(1,0,0)$.
	At any prime $p\equiv 7 \bmod 8$, the $\mu$-ordinary Newton polygon is $u_1=\sss$  
	\cite[Section 4, $m=4$]{LMPT}.
	Then $d=2$ and $r=2$.  By \Cref{Sindprec}, the induced signature type is  
	$\cf_1^\dagger=(1,0,0,0,1,0,0)$.

	The pair of monodromy data for $Z_1$ and $Z_2$ satisfies hypothesis (A).
	Let $p \equiv 7 \bmod 8$; then it also satisfies hypothesis (B).
	For the orbit $\co=\{1,7\} $, by \cite[Example~4.5]{LMPT2}, 
	$u_1(\co) $ has slopes $1/2$ and
	$u_2(\co) $ has slopes $0$ and $1$.
	Thus the pair does not satisfy hypothesis (C). 

	The image of $\tilde{Z}_1\times \tilde{Z}_2$ under the clutching morphism lies in the family $\tilde{Z}_3$ of curves with monodromy datum $(8,5,(2,2,2,5,5))$.
The Shimura variety $\Sh_3$ has signature type $\cf_3=(2,2,0,0,3,1,1)$ and
	its $\mu$-ordinary Newton polygon is $u_3=u_1^2 \oplus u_2 \oplus \ord^2 = \sss^5\oplus \ord^4$ 
	by \Cref{balanced}. By \cite[Section 4.3]{LMPT2}, there is only one element $u_3$ in $B(\Sh_3)$ which is strictly larger than $\sss^7\oplus \ord^2=u_1^2\oplus \nu\oplus \ord^2$. 
	From \eqref{eqn_codim}, we see that the codimension of $\Sh_3[u_1\oplus \nu_2 \oplus \ord^2]$ in $\Sh_3$ is $1$.  Thus, we conclude by \Cref{rmkcodim} and  \Cref{differentm}.
\end{proof}


\section{Supersingular cases in Moonen's table}\label{MoonenSS}

In \cite[Theorem 3.6]{moonen}, Moonen proved there are exactly $20$ positive-dimensional special families arising from cyclic covers of $\PP$. In \cite[Section 6]{LMPT2}, we computed all of the Newton polygons $\nu$ that occur on the corresponding Shimura varieties using the Kottwitz method, see \Cref{TABLE}. Moreover, in \cite[Theorem 1.1]{LMPT2}, we proved that the open Torelli locus intersects each non-supersingular (resp.\ supersingular) Newton polygon stratum
(resp.\ as long as the family has dimension $1$ and $p$ is sufficiently large).

In this section, we extend \cite[Theorem 1.1]{LMPT2} to include the supersingular Newton polygon strata in the five remaining 
cases when the dimension of the family is greater than $1$, using results from \Cref{sec_nonord}. 
Case (5) is note-worthy since it was not previously known that there exists a smooth supersingular curve 
of genus $6$ when $p \equiv 2,3,4 \bmod 5$, see \cite[Theorem 1.1]{LMPT} and \cite[Theorem 1.1]{LMPT2} for related results. 

\begin{theorem}\label{LastCases}
	There exists a smooth supersingular curve of genus $g$ defined over $\overline{\FF}_p$ for all sufficiently large primes 
	satisfying the given congruence condition in the following families:
	\begin{enumerate}
		\item $g=3$, when $p \equiv 2 \bmod 3$, in the family $M[6]$;
		\item $g=3$, when $p \equiv 3 \bmod 4$, in the family $M[8]$;
		\item $g=4$, when $p \equiv 2 \bmod 3$, in the family $M[10]$;
		\item $g=4$, when $p \equiv 5 \bmod 6$, in the family $M[14]$; and
		\item $g=6$, when $p \equiv 2,3,4 \bmod 5$, in the family $M[16]$.
	\end{enumerate}
\end{theorem}

\begin{corollary} \label{MSS}
	Let $\gamma = (m,N,a)$ 
	denote the monodromy datum for one of Moonen's special families from \cite[Table 1]{moonen}.  
	Assume $p \nmid m$.   
	Let $\nu \in \nu(B(\mu_m, \cf))$ be a Newton polygon occurring on $\Sh(\gamma)$ as in \Cref{Skottwitz}. 
	Then $\nu$ occurs as the Newton polygon of a smooth curve in the family $Z^\circ(\gamma)$, 
	as long as $p$ is sufficiently large when $\nu$ is supersingular. 
\end{corollary}

\begin{proof}
	The proof is immediate from \cite[Theorem 1.1]{LMPT2} and Theorem \ref{LastCases}.
\end{proof}

\begin{proof}[Proof of Theorem \ref{LastCases} in cases (1), (2), (4), and (5)]
	Let $\gamma$ denote the monodromy datum, let $Z$ denote the special family of curves and 
	let $\Sh$ denote the corresponding Shimura variety and suppose that $p \not \equiv 1 \bmod m$.
	Then $\dim (Z)=\dim (\Sh)=2$, 
	and the basic locus $\Sh[\nu]$ is supersingular with codimension $1$ in $\Sh$.
	
	Following \cite[Section 5.2]{LMPT2}, a point of $\Sh[\nu]$ which is not in the image of 
	$Z^\circ$ is the Jacobian of a singular curve of compact type. 
	This point arises from an admissible clutching of points from two families $Z_1$ and $Z_2$. 
	This yields an admissible degeneration of the inertia type, see \cite[Definition 5.4]{LMPT2}.
	A complete list of admissible degenerations of the inertia type
	for Moonen's families can be found in \cite[Lemma~6.4]{LMPT2}.  
	In each of these cases, there exists an admissible degeneration such that $\dim (Z_1)=0$ and the $\mu$-ordinary Newton polygon 
	$u_1$ for $Z_1$ is supersingular, and $m_1=m_2$ (so $d=1$). 
	
	In the degenerations from \cite[Lemma~6.4]{LMPT2}, one checks using 
	\cite[Sections 6.1-6.2]{LMPT2} that 
	$Z_2$ is a special family with $\dim (Z_2)=1$ 
	and that $Z_2$ has exactly two Newton polygons, the $\mu$-ordinary one $u_2$ and the basic one $\nu_2$ which is supersingular.
	By \cite[Theorem 1.1]{LMPT2}, for $p$ sufficiently large, $Z_2^\circ[\nu_2]$ is non-empty.
	Since there are exactly two Newton polygons on $Z$, we conclude that these are $u=u_1 \oplus u_2$ and 
	$\nu=u_1 \oplus \nu_2$.
	By Proposition \ref{balanced}, the pair of monodromy data for $Z_1$ and $Z_2$ satisfies hypothesis (B). 
	The codimension condition in \eqref{Ecodimequal} is satisfied since
	the basic locus has codimension $1$ in both $Z$ and $Z_2$.
	By Remark \ref{rmkcodim} and Theorem \ref{differentm}, there exists a 
	1-dimensional family of smooth curves in $Z$ with the basic Newton polygon $\nu$, which is supersingular.
	\end{proof}

	\begin{proof}[Proof of Theorem \ref{LastCases} in case (3)]
	We use the same notation as in the first 2 paragraphs of the proof of the other cases.
	The only difference in case (3) is that $\dim (Z)=\dim (\Sh)=3$ 
	and the basic locus is supersingular with codimension $2$ in $\Sh$.
	In case (3), the only admissible degeneration comes from the pair of monodromy data 
	$\gamma_1=(3,3, (1,1,1))$ and $\gamma_2=(3,5, (2,1,1,1,1))$.
	The latter of these is the monodromy datum for the special family $M[6]$. 
	The basic locus $\Sh[\nu]$ has dimension $1$.
	The codimension condition in \eqref{Ecodimequal} is not satisfied in this situation:
		${\rm codim}(\Sh_2[\nu_2], \Sh_2) =1$, while ${\rm codim}(\Sh[\nu], \Sh) =2$.
	
	For $p$ sufficiently large, we claim that the number of irreducible components of $\Sh[\nu]$
	exceeds the number that arise from the boundary of $Z$. 
	Let $W$ be a $1$-dimensional family of supersingular singular curves in $Z\setminus Z^\circ$. 
	The only way to construct such a family $W$
	is to clutch a genus $1$ curve with $\mu_3$-action together with a $1$-dimensional family of supersingular curves 
	in $M[6]$. 
	In other words, $W$ arises as the image under $\kappa$ of $T_1 \times T_2$, 
	for some component $T_1$ of $\Sh (3,3, (1,1,1))$ and some component $T_2$ of the supersingular locus of $M[6]$. 
	The number of choices for $T_1$, for the $\mu_3$-actions, and for the labelings of the ramification points is 
	a fixed constant that does not depend on $p$.  
	
	Thus it suffices to compare the number $s_{M[10]}$ of irreducible components of the supersingular locus in $M[10]$
	with the number $s_{M[6]}$ 
	of irreducible components $T_2$ of the supersingular locus in $M[6]$ when $p \equiv 2 \bmod 3$.
	The signature type for $M[10]$ is $(1,3)$. 
	By \cite[Theorem 8.1]{LMPT2}, the number $s_{M[10]}$ grows with respect to $p$.
	
	The signature type for $M[6]$ is $(1,2)$.
	By \cite[Remark 8.2]{LMPT2}, we see that $s_{M[6]}$ is the same for all odd $p \equiv 2 \bmod 3$. 
	More precisely, note that $\dim (\Sh_2)=2\dim(\Sh_2(\nu_2))$ when $p\equiv 2 \bmod 3$, that the center of the associated reductive group is connected, and that the supersingular locus is the basic locus.  
	Thus by \cite[Remark 1.1.5 (2)]{LZ}, all odd $p\equiv 2 \bmod 3$ satisfy the hypothesis of \cite[Theorem~1.1.4 (1), Proposition~7.4.2]{LZ}, which provides an expression for $s_{M[6]}$ over $\overline{\FF}_p$ in terms of objects independent of $p$.

	Hence there exist irreducible components of $\Sh[\nu]$ which 
	contain the Jacobian of a smooth curve, for $p$ sufficiently  large.
\end{proof}


\section{Unlikely intersections}\label{unlikely}

In this section, we prove that the non-trivial intersection of the open Torelli
locus with the Newton polygon strata found in most of the results of the paper is unexpected.

Recall that $\sss$ denotes the Newton polygon $(1/2,1/2)$ . 

\begin{definition}\label{def_conditionU}
Let $\nu$ be a symmetric Newton polygon of height $2g$, 
and let $\mathcal{A}_g[\nu]$ be its Newton polygon stratum in the Siegel variety $\mathcal{A}_g$. 
	Then $\nu$ satisfies condition (U) if 
	${\rm dim}({\mathcal M}_g) < {\rm codim}({\mathcal A}_g[\nu], {\mathcal A}_g)$.
\end{definition}

\begin{definition}\label{def_unlikely}
	The open Torelli locus has an \emph{unlikely intersection} with 
	${\mathcal A}_g[\nu]$ in ${\mathcal A}_g$ if there exists a smooth curve of genus $g$ with Newton polygon $\nu$,
	and $\nu$ satisfies condition (U).
\end{definition}

\subsection{The codimension of Newton polygon strata in Siegel varieties}

We study the codimension of the Newton strata in ${\mathcal A}_g$. 
By \cite[Theorem 4.1]{oort01}, see also \eqref{eqn_codim}, 
\begin{equation}\label{codimNewton}
\codim (\mathcal{A}_g[\nu], \mathcal{A}_g) = \# \Omega(\nu),
\end{equation}
where $\Omega(\nu) := \{ (x,y) \in \mathbb{Z} \times \mathbb{Z} \mid 0 \le x,y \le g, \ (x,y) \text{ strictly below } \nu \}$.

\begin{remark}\label{QuadraticGrowth}
	By \eqref{codimNewton}, if $\nu$ is non-ordinary, then 
	${\rm codim}(\mathcal{A}_{ng}[\nu^n], \mathcal{A}_{ng})$ grows quadratically in $n$.
	In particular, if $\nu=\sss$, then $\codim (\mathcal{A}_{n}[\sss^n] , \mathcal{A}_n)= n(n+1)/2-\lfloor n^2/4 \rfloor >n^2/4$.
\end{remark}

\begin{proposition}\label{QuadraticGrowth2}
	Let $\{u_n\}_{n\in {\mathbb N}}$ be a sequence of symmetric Newton polygons.
	Let $2g_n$ be the height of $u_n$. 
	Suppose there exists $\lambda\in {\mathbb Q}\cap (0,1)$ such that
	the multiplicity of $\lambda$ as a slope of $u_n$ is at least $n$ for each $n\in {\mathbb N}$. 
	Then ${\rm codim}(\mathcal{A}_{g_n}[u_n], \mathcal{A}_{g_n})$ grows at least quadratically in $n$.
\end{proposition}

\begin{proof}
	Let $\nu=(\lambda,1-\lambda)$ and let $h$ be the height of $\nu$. 
	By hypothesis, $u_n=\nu^n\oplus \nu_n$ for some symmetric Newton polygon $\nu_n$ for each $n\in {\mathbb N}$ and $g_n\geq nh$. 
Since $\nu_n$ lies on or above $\ord^{g_n-nh}$, then $u_n=\nu^n\oplus \nu_n$ lies on or above $\nu^n \oplus \ord^{g_n-nh}$.  Hence 
$$\codim (\mathcal{A}_{g_n}[u_n],\mathcal{A}_{g_n}) \ge \codim (\mathcal{A}_{g_n}[\nu^n \oplus \ord^{g_n-nh}],\mathcal{A}_{g_n}).$$ 
By \eqref{codimNewton}, or alternatively \Cref{codim}, 
\[ \codim (\mathcal{A}_{g_n}[\nu^n \oplus \ord^{g_n-nh}],\mathcal{A}_{g_n}) \ge \codim (\mathcal{A}_{nh}[\nu^n],\mathcal{A}_{nh}).\]
Thus $\codim (\mathcal{A}_{g_n}[u_n],\mathcal{A}_{g_n}) \ge \codim (\mathcal{A}_{nh}[\nu^n],\mathcal{A}_{nh})$, 
which is sufficient by \Cref{QuadraticGrowth}. \end{proof}

\subsection{Verifying condition (U)}

Given a sequence $\{u_n\}_{n\in{\mathbb N}}$ of symmetric Newton polygons of
increasing height, we state simple criteria to ensure that 
all but finitely many of them satisfy condition (U). 
Let $2g_n$ be the height of $u_n$. 

\begin{proposition}\label{thm_unlikely}
	Assume that $g_n$ grows linearly in $n$ and that there exists $\lambda\in {\mathbb Q}\cap (0,1)$ such that the multiplicity of $\lambda$ as a slope of $u_n$ grows linearly in $n$, for all sufficiently large $n\in {\mathbb N}$.  Then, for all sufficiently large $n$, the Newton polygon $u_n$ satisfies condition (U). 	
\end{proposition}

\begin{proof}
	By \Cref{QuadraticGrowth2},  $\codim(\mathcal{A}_{g_n}[u_n], \mathcal{A}_{g_n})$ is quadratic in $n$ 
	while $\dim(\mathcal{M}_{g_n})= 3g_n-3$ is linear in $n$ by hypothesis. 
	Thus $\dim (\mathcal{M}_{g_n}) < \codim(\mathcal{A}_{g_n}[u_n],\mathcal{A}_{g_n})$ for $n\gg 0$.
\end{proof}

\begin{proposition}\label{Codimforss}
	If there exists $t\in {\mathbb R}_{>0}$ such that the multiplicity of $1/2$ as a slope of $u_n$ is at least $2tg_n$, for all $n\in {\mathbb N}$, then $u_n$ satisfies condition (U) for each $n \in {\mathbb N}$ such that  $g_n \ge 12/t^2$. 
\end{proposition}

\begin{proof}
	By the proof of \Cref{QuadraticGrowth2} and \Cref{QuadraticGrowth},
	$$ \codim (\mathcal{A}_{g_n}[u_n],\mathcal{A}_{g_n})\ge \codim (\mathcal{A}_{\lceil tg_n\rceil}[\sss^{\lceil tg_n\rceil}],\mathcal{A}_{\lceil tg_n\rceil}) 
	> (tg_n)^2/4.$$
	So condition (U) for $u_n$ is true when $(tg_n)^2/4 \ge (3g_n-3)$ and thus when $g_n \ge 12/t^2$.
\end{proof}

\begin{proposition}\label{coro_unlikely}
	Let $\nu_1,\nu_2$ be two symmetric Newton polygons, respectively of height $2g\geq 2$, and $2h\geq 0$. Assume $\nu_1$ is not ordinary. Then
	
	\begin{enumerate}
		\item for all sufficiently large $n\in \NN$, the Newton polygon $\nu_1^n\oplus \nu_2$ satisfies condition (U);
		
		\item if $1/2$ occurs as a slope of $\nu_1$ with multiplicity $2\delta>0$, then the Newton polygon $\nu_1^n\oplus \nu_2$ satisfies condition (U), for each
		$n\geq \max\{15g/\delta^2, 9\sqrt{h}/\delta\}$.\footnote{This bound is not sharp, but it is written so that the asymptotic dependency on $g,\delta, h$ is more clear.}
	\end{enumerate}
\end{proposition}

\begin{proof}
	\begin{enumerate}
		\item
		Let $\lambda\in {\mathbb Q}\cap (0,1)$ be a slope of $\nu_1$, occurring with multiplicity $m_\lambda \geq 1$. 
		Then, for each $n\in {\mathbb N}$, the Newton polygon $u_n=\nu_1^n \oplus \nu_2$ has height $2g_n=2(ng+h)$ and slope $\lambda$ occurring with multiplicity at least $m_\lambda n$. 
		Taking $u_n=\nu_1^n \oplus \nu_2$, the sequence $\{u_n\}_{n \in \NN}$ satisfies the hypotheses of \Cref{thm_unlikely}.  Hence, part (1) holds.
		
		\item
		As for \Cref{QuadraticGrowth2}, $\codim(\mathcal{A}_{ng}[\nu_1^n],\mathcal{A}_{ng})\le \codim(\mathcal{A}_{ng+h}[\nu_1^n \oplus \nu_2],\mathcal{A}_{ng+h})$. Therefore, condition (U) for $\nu_1^n\oplus \nu_2$
		is implied by the inequality \begin{equation}\label{ineq}
		\dim(\mathcal{M}_{ng+h})< \codim(\mathcal{A}_{ng}[\nu_1^n],\mathcal{A}_{ng}).
		\end{equation}
		
		Following the proof of \Cref{Codimforss}, if the slope $1/2$ occurs in $\nu_1$ with multiplicity $2\delta$, 
		then inequality \eqref{ineq} is true if $3(ng+h-1) \le (n\delta)^2/4$,
		which holds for $n\ge N:= 6g\delta^{-2}(1+(1+\delta^2(h-1)3^{-1}g^{-2})^{1/2})$.
		The asserted bound follows by noticing that 
		$N < \max\{6(1+\sqrt{2})g/\delta^2, 2\sqrt{3}(1+\sqrt{2})\sqrt{h}/\delta\}$. \qedhere
	\end{enumerate}
\end{proof}

\begin{remark}\label{one}
For $g\gg 0$, 
	\Cref{coro_unlikely} implies that the non-trivial intersections of ${\mathcal T}^\circ_g$ with
	 $\CA_g[\nu]$ in Corollaries \ref{infinite-ord},
	\ref{inf-muord}, and \ref{inf-nu} (resp.\ \ref{twofamilyinfinite} and \ref{twofamilynext}) 
	are unlikely if the $\mu$-ordinary Newton polygon $u$ is not ordinary.
	(resp.\ if either $u_1$ or $u_2$ is not ordinary).  
	\end{remark}

\begin{remark}\label{two}
	Consider the following refinement of Definition \ref{def_unlikely}:
	a non-empty substack $U$ of ${\mathcal T}_g \cap \mathcal{A}_g[\nu]$ is an {\em unlikely intersection}
	if $\codim(U,\mathcal{M}_g)<\codim(\mathcal{A}_g[\nu],\mathcal{A}_g).$
	
	The results in Sections \ref{Sinfclut} and \ref{sec_infclut_nu}
	yield  families $Z$ of cyclic covers of ${\mathbb P}^1$ such that $Z^\circ[\nu]$ is non-empty and has the expected codimension in $Z$.  
	This produces an unlikely intersection as in Remark \ref{two} for $g\gg 0$, when the initial Newton polygon
	$u$ is not ordinary. 
\end{remark}


\section{Applications}\label{Applications1}

We apply the results in Sections \ref{Sinfclut} and \ref{sec_infclut_nu}
to construct smooth curves of arbitrarily large genus $g$ with prescribed Newton polygon $\nu$. 
By \Cref{coro_unlikely}, when $g$ is sufficiently large, the curves in this section
lie in the unlikely intersection ${\mathcal T}_g^\circ \cap {\mathcal A}_g[\nu]$.  

\begin{notation} For $s,t \in \NN$, with $s\leq t/2$ and $\gcd(s,t)=1$, we write $(s/t,(t-s)/t)$ for the Newton polygon 
of height $2t$ with slopes $s/t$ and $(t-s)/t$, each with multiplicity $t$.
\end{notation}

\subsection{Newton polygons with many slopes of $1/2$} 

We obtain examples of smooth curves of arbitrarily large genus $g$
such that the only slopes of the Newton polygons are $0,\frac{1}{2},1$. 
We focus on examples where the multiplicity of $1/2$ is large relative to $g$.

\begin{corollary}\label{3gn-2g}
	Let $m \in \ZZ_{\ge 1}$ be odd and $h=(m-1)/2$.  
	Let $p$ be a prime, $p \nmid 2m$, such that the order $f$ of $p$ in $(\ZZ/m\ZZ)^*$ is even and $p^{f/2}\equiv -1 \bmod m$.
	For $n \in \ZZ_{\ge 1}$,
	there exists a $\mu_m$-cover $C \to {\mathbb P}^1$ defined over $\overline{\mathbb{F}}_p$
	where $C$ is a smooth curve of genus $g=h(3n-2)$ 
	with Newton polygon $\nu=ss^{h n}\oplus ord^{2h(n-1)}$.
	If $n \ge 34/h$, then ${\rm Jac}(C)$ lies in the unlikely intersection ${\mathcal T}^\circ_g \cap 
	{\mathcal A}_{g}[\nu]$.  
\end{corollary}

\begin{proof}
Let $C \to \PP$ be a $\mu_m$-cover with $\gamma=(m, 3, a)$ where $a=(1,1,m-2)$.  
Without loss of generality, an equation for $C$ is $y^m= x^2 -1$.
By \cite[Theorem 6.1]{yui}, 
the Newton polygon of $C$ is $\sss^h$. 	
The first claim follows from applying \Cref{inf-muord} to $Z(m,3,\gamma)$ with $t=m$.
	As in the proof of \Cref{Codimforss}, the second claim follows from the inequalities:
	\[\codim (\mathcal{A}_{nh}[\sss^{nh}],\mathcal{A}_{nh}) \ge (nh)^2/4+(nh)/2 > \dim (\mathcal{M}_{3nh-2h})=9nh-6h-3.
	\qedhere\] 
	\end{proof}
	
	\begin{remark}
	The Newton polygons in \Cref{3gn-2g} are $\mu$-ordinary; they do not appear in the literature, 
		but the result also follows from \Cref{muordoccur}(3) if $p \equiv -1 \bmod m$ or
	if $p \geq m(N-3)$ where $N$ is the (increasingly large) 
	number of branch points.
\end{remark}

We highlight the case $m=3$ below.  To our knowledge, for any odd prime $p$, 
this is the first time that a sequence of smooth curves has been produced for every $g \in \ZZ_{\ge 1}$ 
such that the multiplicity of the slope $1/2$ in the Newton polygon grows linearly in $g$.

\begin{corollary}\label{2/3ForAllg}
	Let $p \equiv 2 \bmod 3$ be an odd prime.  Let $g \in \ZZ_{\ge 1}$.
	There exists a smooth curve $C_g$ of genus $g$ 
	defined over $\overline{\FF}_p$, whose Newton polygon $\nu_g$
	only has slopes $0, \frac{1}{2}, 1$ and such that the multiplicity of the slope $1/2$ is at least 2$\lfloor g/3 \rfloor$. 
	If $g \geq 107$,
	the curve $C_g$ demonstrates an unlikely intersection
	of the open Torelli locus with the Newton polygon stratum ${\mathcal A}_{g}[\nu_g]$ in ${\mathcal A}_{g}$.
\end{corollary}

\begin{proof}
	If $g = 3n-2$ for some $n$, the result is immediate from \Cref{3gn-2g}.
	For $g =3n -2 + 2 \epsilon$ with $\epsilon = 1$ (resp.\ $\epsilon = 2$), we apply \Cref{inf-ord} with $t=1$ (resp.\ twice)  
	and obtain a smooth curve with Newton polygon $\sss^{n}\oplus \ord^{2n-2+2\epsilon}$. 
\end{proof}

Working with Moonen's families gives examples of families of curves where the multiplicity of the slope $1/2$ 
is particularly high relative to the genus.

\begin{corollary}\label{2/5for4mod5}
	Let $p \equiv 4 \bmod 5$. 
		For $n \in {\mathbb Z}_{\geq 1}$,
		there exists a smooth curve of genus $g=10n-4$ in $Z=Z(5, 5n, (2,2,\dots ,2))$ over  $\overline{\FF}_p$
		with $\mu$-ordinary Newton polygon $u_n= \sss^{4n} \oplus \ord^{6n-4}$.
\end{corollary}

For $n\geq 7$, 
		the curves with Newton polygon $u_n$ from \Cref{2/5for4mod5} 
		lie in the unlikely intersection ${\mathcal T}_g^\circ \cap \CA_g[\nu]$.

\begin{proof}
When $p \equiv 4 \bmod 5$, $M[16]$ has $\mu$-ordinary Newton polygon $u_1=\ord^2\oplus \sss^4$.\footnote{The codimension condition in \eqref{Ecodimequal} does not hold for $\nu=\sss^6$.}
The claim is immediate from \Cref{inf-muord}. 	\end{proof}

\begin{corollary}\label{table}
	Under the given congruence condition on $p$, and with $p \gg 0$, 
	there exists a smooth curve in $Z=Z(m, N, a)$ 
	over $\overline{{\mathbb F}}_p$ 
	with Newton polygon $\nu$ and ${\rm codim}(Z[\nu], Z)=1$.
	\begin{center}
		\begin{tabular}{ |c|c|c|c|c|c| }
			\hline
			construction &(m,N,a) &	genus & congruence & Newton Polygon $\nu$ \\
			\hline
			
			$M[9] + M[9]$ &$(6,6,(1,1,4,4,4,4))$ & 8 & $2 \bmod 3$ & $\sss^4\oplus \ord^4$\\
			\hline
			
			$M[9] + M[12]$ &$(6,6,(1,1,1,1,4,4))$ & 9 & $2 \bmod 3$ & $\sss^5\oplus \ord^4$\\
			\hline
			
			$M[12] + M[12]$ &$(6, 6, (1,1,1,1,1,1))$ & 10 & $2 \bmod 3$ & $\sss^7 \oplus \ord^3$ \\
			
			
			\hline		
			$M[18] + M[18]$ &$(10,6,(3,3,6,6,6,6))$ & 16 & $9 \bmod 10$ & $\sss^{10}\oplus \ord^6$\\
			\hline
			
			$M[20] + M[20]$ &$(12,6,(4,4,7,7,7,7))$ & 19 & $11 \bmod 12$ & $\sss^{12}\oplus \ord^7$\\
			\hline
			
		\end{tabular}
	\end{center}
\end{corollary}

\subsection{Newton polygons with slopes $1/3$, $1/4$, and beyond}

\begin{corollary} \label{Cslope13}\label{table2}
Let $n \in \mathbb{Z}_{\geq 1}$.
The following Newton polygons occur for Jacobians of smooth curves over $\overline{\FF}_p$ 
under the given congruence condition on $p$.
	\begin{center}
		\begin{tabular}{ |c|c|c| }
			\hline
			 congruence & $\nu$ ($\mu$-ordinary) & $\nu$ (non $\mu$-ordinary)
			 for $p \gg 0$ \\
			\hline
$2,4 \bmod 7$ & $(1/3,2/3)^n \oplus \ord^{6n-6}$ & NA \\
\hline

			$3,5 \bmod 7$ & 
			$(1/3,2/3)^{2n}\oplus \ord^{6n-6}$ & $(1/3,2/3)^{2n-2} \oplus \sss^6 \oplus \ord^{6n-6}$  \\
			\hline
			
			
			$2,5 \bmod 9$ & 
			$(1/3,2/3)^{2n} \oplus \sss^{n}  \oplus \ord^{8n-8}$ & $(1/3,2/3)^{2n-2} \oplus \sss^{n+6}  \oplus \ord^{8n-8}$  \\	
			\hline 

			$4,7 \bmod 9$ & 
			$(1/3, 2/3)^{2n} \oplus \ord^{9n-8}$ & $(1/3, 2/3)^{2n-2} \oplus \sss^6 \oplus \ord^{9n-8}$	\\			
			\hline
			
		\end{tabular}
	\end{center}
\end{corollary}

We remark that none of the last three lines follows from \cite[Theorem 6.1]{Bouwprank} because there are at least two Newton polygons 
in $B(\mu_m, \cf)$ having the maximal $p$-rank.

\begin{proof}
Lines 1, 2, and 3 are obtained from applying 
both Corollaries \ref{inf-muord} and \ref{inf-nu} to the families $(7, 3, (1,1,5))$, $M[17]$, and $M[19]$, 
respectively. 

For the last line, let $m=9$ and $p \equiv 4,7 \bmod 9$.
There are four orbits $\co_1=(1,4,7)$, $\co_2=(2,5,8)$, $\co_3=(3)$, and $\co_4=(6)$. 
The $\mu$-ordinary Newton polygon for the family $M[19]$ is $u=(1/3,2/3)^2\oplus \ord$, and $\nu=\sss^6\oplus \ord$ also occurs for a smooth curve in the family. By \cite[Section 6.2]{LMPT2}, for each $\co\in\CO$, $u(\co)$ has at most 2 slopes, hence hypothesis (C) is satisfied, and  we obtain the Newton polygons in line 4 from \Cref{inf-muord,inf-nu}.\footnote{Alternatively, 
applying  \Cref{twofamilyinfinite,twofamilynext}
produces the Newton polygons $(1/3,2/3)^{n_1+2n_2}\oplus \ord^{8n_1+9n_2-14}$ 
	and $(1/3,2/3)^{n_1+2n_2-2} \oplus  \sss^6\oplus \ord^{8n_1+9n_2-13}$ for $n_1,n_2 \in \ZZ_{\geq 1}$.}
\end{proof}

\begin{corollary} \label{Cslope14}
Let $n \in \mathbb{Z}_{\geq 1}$.
The following Newton polygons occur for Jacobians of smooth curves over $\overline{\FF}_p$ 
under the given congruence condition on $p$.
	\begin{center}
		\begin{tabular}{ |c|c|c| }
			\hline
			 congruence & $\nu$ ($\mu$-ordinary) & $\nu$ (non $\mu$-ordinary)
			 for $p\gg 0$ \\
			\hline
$2,3 \bmod 5$ & $(1/4,3/4)^{n}\oplus \ord^{4n-4}$ & $(1/4,3/4)^{n-1}\oplus \sss^4\oplus \ord^{4n-4}$ \\ \hline
$3,7 \bmod 10$ & $(1/4,3/4)^n \oplus \sss^{2n}  \oplus \ord^{9n-9}$ & $(1/4,3/4)^{n-1} \oplus \sss^{2n+4}  \oplus \ord^{9n-9}$ \\
			\hline
		\end{tabular} \end{center}
		\end{corollary}

\begin{proof}
	The proof follows from \Cref{inf-muord,inf-nu} applied to $M[11]$ and $M[18]$.
	\end{proof}

\begin{corollary}\label{corolast}\label{n2vsn1}
	Let $p \equiv 2,3 \bmod 5$. 
	For any $n_1,n_2 \in {\mathbb Z}_{\ge 1}$,
	there exists a smooth curve of genus $g=6n_1+8n_2$ defined over $\overline{\FF}_p$ with Newton polygon
	$(1/4,3/4)^{n_2+1}\oplus \sss^{2n_1} \oplus \ord^{4(n_1+n_2-1)}$.  	
\end{corollary}

\begin{proof}
	We apply \Cref{twofamilyinfinite} to $Z_1=Z(5,3,(2,2,1))$ and $Z_2=M[11]$.	
	By \cite[Section 6.2]{LMPT2} and \cite[Section 4]{LMPT},
	if $p \equiv 2,3\bmod 5$, then $u_1=\sss^2$ and $u_2=(1/4,3/4)$.
	\footnote{The pair $Z_1$ and $Z_2$ does not satisfy hypothesis (C) and the codimension condition in \eqref{Ecodimequal} does not hold inductively.}	 
\end{proof}

\begin{example} \label{Eno12}
Let $m$ be prime and $p$ have odd order modulo $m$.
The Newton polygon $\nu_1$ for a $\mu_m$-cover with monodromy datum $\gamma=(m, 3, a)$
has no slopes of $1/2$ by \cite[Section~3.2]{LMPT}.
Applying \Cref{inf-muord} to $Z=Z(\gamma)$ with $t=m$ shows that 
the Newton polygon $\nu_n=\nu_1^n \oplus \ord^{(m-1)(n-1)}$ occurs
for a smooth curve over $\overline{\FF}_p$, for any $n \in {\mathbb Z}_{\geq 1}$. 
\end{example}

Examples of $\gamma$ and $\nu_1$ can be found in \cite[Theorem 5.4]{LMPT}.
For example, when $m=11$, $a=(1,1,9)$ and $p \equiv 3,4,5,9 \bmod 11$, then $\nu_1=(1/5,4/5)$.
As another example, let $m=29$, $a=(1,1,27)$, and $p \equiv 7,16,20,23,24,25 \bmod 29$, 
then $\nu_1=(2/7,5/7)\oplus (3/7,4/7)$, 
yielding another infinite family that cannot be studied using \cite[Theorem~6.1]{Bouwprank}.

\section{Appendix: Newton polygons for Moonen's families}\label{TABLE}

For convenience, we provide the full list of Newton polygons on Moonen's special families 
from \cite[Section 6]{LMPT2}.  These occur for a smooth curve in the family by 
\Cref{MSS}.
The label $M[r]$ is from \cite[Table 1]{moonen}.
The notation $^\dagger$ means {\bf we further assume $p \gg 0$}.

\newpage
\begin{small}
	\begin{center}
		\begin{tabular}{  |c|c|c|c|  }
			\hline
			Label & m& a, $\cf$ & Newton Polygon [congruence on $p$]\\
			\hline
			$M[1]$& 2 & (1,1,1,1), $\cf= (1)$ & $\ord$, $\sss$ $[1 \bmod 2]$ \\
			\hline
			$M[2]$& 2 & (1,1,1,1,1,1), $\cf= (2)$ & $\ord^2$, $\ord \oplus \sss$, $\sss^2$ $[1 \bmod 2]$ \\
			\hline
			$M[3]$& 3 & (1,1,2,2), $\cf= (1,1)$  & $\ord^2$ $[1,2 \bmod 3] $, $\sss^2$ $[1 \bmod 3] $, $\sss^2$ $[2 \bmod 3]^\dagger $\\
			\hline
			$M[4]$ &4 & (1,2,2,3)  &$\ord^2$ $[1,3 \bmod 4] $\\
			&& $\cf= (1,0,1)$ & $\sss^2$ $[1 \bmod 4]$, $\sss^2$ $[3 \bmod 4]^\dagger $\\
			\hline
			$M[5]$ & 6& (2,3,3,4)  & $\ord^2$ $[1,5 \bmod 6] $\\
			&& $\cf= (1,0,0,0,1)$ & $\sss^2$ $[1 \bmod 6]$, $\sss^2$ $[5 \bmod 6]^\dagger $\\
			\hline
			$M[6]$ & 3 & (1,1,1,1,2) & $ \ord^3$ $[1 \bmod 3] $, $\ord^2 \oplus \sss$ $[2 \bmod 3] $\\
			&& $\cf= (2,1)$ & $\ord \oplus \sss^2$, $(1/3,2/3)$ $[1 \bmod 3] $, $\sss^3$ $[2 \bmod 3]^\dagger $\\
			\hline
			$M[7]$ & 4 & (1,1,1,1) & $\ord^3$ $[1 \bmod 4] $, $\ord \oplus \sss^2$ $[3 \bmod 4] $ \\
			&& $\cf= (2,1,0)$ & $\ord^2 \oplus \sss$ $[1 \bmod 4] $, $\sss^3$ $[3 \bmod 4]^\dagger $ \\
			\hline
			$M[8]$ & 4 & (1,1,2,2,2) & $ \ord^3$ $[1 \bmod 4] $, $\ord^2 \oplus \sss$ $[3 \bmod 4] $\\
			&& $\cf= (2,0,1)$ & $\ord \oplus \sss^2$, $(1/3,2/3)$ $[1 \bmod 4] $, $\sss^3$ $[3 \bmod 4]^\dagger $\\
			\hline
			$M[9]$ & 6 &(1,3,4,4)& $\ord^3$ $[1 \bmod 6] $, $\ord^2 \oplus \sss$ $[5 \bmod 6] $\\
			&& $\cf= (1,1,0,0,1)$ & $\ord \oplus \sss^2$ $[1 \bmod 6] $, $\sss^3$ $[5 \bmod 6]^\dagger $\\
			\hline
			$M[10]$& 3 & (1,1,1,1,1,1)& $ \ord^4$ $[1 \bmod 3] $, $\ord^2 \oplus \sss^2$ $[2 \bmod 3] $\\
			&& $\cf= (3,1)$ & $\ord^2 \oplus \sss^2$ $[1 \bmod 3] $, $(1/4,3/4)$ $[2 \bmod 3] $\\
			&&& $\ord \oplus (1/3,2/3)$ $[1 \bmod 3] $, $\sss^4$ $[2 \bmod 3]^\dagger $\\
			&&& $(1/4,3/4)$ $[1 \bmod 3] $\\
			\hline
			$M[11]$& 5 & (1,3,3,3) &  $\ord^4$ $[1 \bmod 5] $, $(1/4,3/4)$ $[2,3 \bmod 5] $, $\ord^2 \oplus \sss^2$ $[4 \bmod 5] $\\ 
			&& $\cf= (1,2,0,1)$ &$\ord^2 \oplus \sss^2$ $[1 \bmod 5] $, $\sss^4$ $[2,3,4 \bmod 5]^\dagger $\\
			\hline
			$M[12]$ &6& (1,1,1,3) & $\ord^4$ $[1 \bmod 6] $, $\ord \oplus \sss^3$ $[5 \bmod 6] $\\ 
			& & $\cf= (2,1,1,0,0)$ &  $\ord^3 \oplus \sss$ $[1 \bmod 6] $, $\sss^4$ $[5 \bmod 6]^\dagger $\\
			\hline
			$M[13]$ & 6&(1,1,2,2) & $\ord^4$ $[1 \bmod 6] $, $\ord^2 \oplus \sss^2$ $[5 \bmod 6] $\\
			&& $\cf= (2,1,0,1,0)$ & $ \ord^2 \oplus \sss^2$ $[1 \bmod 6] $, $\sss^4$ $[5 \bmod 6]^\dagger $\\
			\hline
			$M[14]$ & 6& (2,2,2,3,3) & $ \ord^4$ $[1 \bmod 6] $, $ \ord^2 \oplus \sss^2$ $[5 \bmod 6] $\\
			&& $\cf= (2,0,0,1,1)$ & $ \ord^2 \oplus \sss^2$ $[1 \bmod 6] $, $ \sss^4$ $[5 \bmod 6]^\dagger $\\
			&&& $\ord \oplus (1/3,2/3)$ $[1 \bmod 6] $\\
			\hline
			$M[15]$& 8 & (2,4,5,5), $\cf=$ & $\ord^5$ $[1 \bmod 8] $, $\ord^2 \oplus \sss^3$ $[3,7 \bmod 8] $, $\ord^3 \oplus \sss^2$ $[5 \bmod 8] $\\  
			&& $(1,1,0,0,2,0,1) $ &  $\ord^3 \oplus \sss^2$ $[1 \bmod 8]$, $(1/4,3/4) \oplus \sss$ $[3 \bmod 8]$\\
			&&  & $\ord \oplus (1/4,3/4)$ $[5 \bmod 8] $, $\sss^5$ $[7 \bmod 8]^\dagger $\\
			\hline
			$M[16]$ & 5 & (2,2,2,2,2) & $ \ord^6$ $[1 \bmod 5] $, $ (1/4,3/4) \oplus \sss^2$ $[2,3 \bmod 5] $, $ \ord^2 \oplus \sss^4$ $[4 \bmod 5] $\\
			&& $\cf= (2,0,3,1)$ & $ \ord^4 \oplus \sss^2$ $[1 \bmod 5]$, $ \sss^6$ $[2,3,4 \bmod 5]^\dagger $\\
			&&& $\ord^3 \oplus (1/3,2/3)$ $[1 \bmod 5]$\\
			\hline
			$M[17]$ & 7 & (2,4,4,4), $\cf=$ & $\ord^6$ $[1 \bmod 7] $, $\ord^3 \oplus (1/3,2/3)$ $[2,4 \bmod 7] $\\ 
			& &  $(1,2,0,2,0,1)$ & $(1/3,2/3)^2$ $[3,5 \bmod 7] $, $\ord^2 \oplus \sss^4$ $[6 \bmod 7] $ \\ 
			&& &  $\ord^4 \oplus \sss^2$ $[1 \bmod 7]$, $(1/6,5/6)$ $[2,4 \bmod 7] $, $\sss^6$ $[3,5,6 \bmod 7]^\dagger $\\
			\hline
			$M[18]$ & 10 &(3,5,6,6), $\cf=$& $\ord^6$ $[1 \bmod 10] $, $(1/4,3/4) \oplus \sss^2$ $[3,7 \bmod 10] $, $\ord^2 \oplus \sss^4$ $[9 \bmod 10] $\\  
			&& $(1,1,0,1,0,0,2,0,1)$ &  $\ord^4 \oplus \sss^2$ $[1 \bmod 10] $, $\sss^6$ $[3,7,9 \bmod 10]^\dagger $\\
			\hline
			$M[19]$ & 9 & (3,5,5,5),  $\cf=$& $\ord^7$ $[1 \bmod 9] $, $(1/3,2/3)^2 \oplus \sss$ $[2,5 \bmod 9] $\\ 
			&& $(1,2,0,2,0,1,0,1)$  &$\ord \oplus (1/3,2/3)^2$ $[4,7 \bmod 9] $, $\ord^2 \oplus \sss^5$ $[8 \bmod 9] $\\
			&& &  $\ord^5 \oplus \sss^2$ $[1 \bmod 9]$, $\sss^7$ $[2,5,8 \bmod 9]^\dagger$, $\ord \oplus \sss^6$ $[4,7 \bmod 9] $\\
			\hline
			$M[20]$ & 12 & (4,6,7,7), $\cf=$& $\ord^7$ $[1 \bmod 12] $, $\ord^3 \oplus \sss^4$ $[5 \bmod 12] $\\
			&& $(1,1,0,1,0,0, $ & $\ord^4 \oplus \sss^3$ $[7 \bmod 12] $, $\ord^2 \oplus \sss^5$ $[11 \bmod 12] $\\
			&& $2,0,1,0,1) $ &  $\ord^5 \oplus \sss^2$ $[1 \bmod 12] $, $\ord \oplus (1/4,3/4) \oplus \sss^2$ $[5 \bmod 12] $\\
			&&&$\ord^2 \oplus \sss^5$ $[7 \bmod 12] $, $\sss^7$ $[11 \bmod 12]^\dagger $\\
			\hline
		\end{tabular}
	\end{center}
\end{small}

\bibliographystyle{amsplain}
\bibliography{npthirdbib}

\end{document}